\documentclass[11pt]{article}
\usepackage[lmargin=30mm,rmargin=30mm,bmargin=30mm,tmargin=30mm]{geometry}
\usepackage[colorlinks=true,citecolor=black,linkcolor=black,urlcolor=blue]{hyperref}
\usepackage{amsmath,amsthm,amsfonts,amssymb,enumerate,calc,graphicx,xcolor,tabu}
\usepackage{iwona}
\usepackage[noabbrev,capitalise]{cleveref}
\usepackage[nodayofweek]{datetime}
\longdate
\usepackage[longnamesfirst,numbers,sort&compress]{natbib}
\renewcommand{\baselinestretch}{1.08}
\usepackage[hang]{footmisc} 

\renewcommand{\thefootnote}{\fnsymbol{footnote}}	
\newcommand\DateFootnote{
\begingroup
\renewcommand\thefootnote{}
\footnote{27$^{\text{th}}$ July, 2007. Revised: \today}
\setcounter{footnote}{0}
\vspace*{-3ex}
\endgroup}

\makeatletter
\renewcommand\section{\@startsection {section}{1}{\z@}%
                                   {-3ex \@plus -1ex \@minus -.2ex}%
                                   {2ex \@plus.2ex}%
                                   {\normalfont\large\bfseries}}
\renewcommand\subsection{\@startsection{subsection}{2}{\z@}%
                                     {-2.5ex\@plus -1ex \@minus -.2ex}%
                                     {1.5ex \@plus .2ex}%
                                     {\normalfont\normalsize\bfseries}}
\renewcommand\subsubsection{\@startsection{subsubsection}{3}{\z@}%
                                     {-2ex\@plus -1ex \@minus -.2ex}%
                                     {1ex \@plus .2ex}%
                                     {\normalfont\normalsize\bfseries}}
 \renewcommand\paragraph{\@startsection{paragraph}{4}{\z@}%
                                    {1.5ex \@plus.5ex \@minus.2ex}%
                                    {-1em}%
                                    {\normalfont\normalsize\bfseries}}
\renewcommand\subparagraph{\@startsection{subparagraph}{5}{\parindent}%
                                       {1.5ex \@plus.5ex \@minus .2ex}%
                                       {-1em}%
                                      {\normalfont\normalsize\bfseries}}
\makeatother

\newcommand{\arXiv}[1]{arXiv:\,\href{http://arxiv.org/abs/#1}{#1}}
\newcommand{\msn}[1]{MR:\,\href{http://www.ams.org/mathscinet-getitem?mr=MR#1}{#1}}

\newcommand{\doi}[1]{doi:\,\href{http://dx.doi.org/#1}{#1}}

\theoremstyle{plain}
\newtheorem{thm}{Theorem}
\newtheorem{lem}[thm]{Lemma}
\newtheorem{conj}[thm]{Conjecture}
\newtheorem{cor}[thm]{Corollary}
\newtheorem{prop}[thm]{Proposition}

\newcommand{\seclabel}[1]{\label{sec:#1}}

\newcommand{\secref}[1]{\mbox{Section~\ref{sec:#1}}}
\newcommand{\lemlabel}[1]{\label{lem:#1}}
\newcommand{\lemref}[1]{Lemma~\ref{lem:#1}}
\newcommand{\twolemref}[2]{Lemmas~\ref{lem:#1} and~\ref{lem:#2}}
\newcommand{\thmlabel}[1]{\label{thm:#1}}
\newcommand{\thmref}[1]{Theorem~\ref{thm:#1}}

\newcommand{\corlabel}[1]{\label{cor:#1}}
\newcommand{\corref}[1]{Corollary~\ref{cor:#1}}
\newcommand{\figlabel}[1]{\label{fig:#1}}

\newcommand{\figref}[1]{\mbox{Figure~\ref{fig:#1}}}

\newcommand{\eqnlabel}[1]{\label{eqn:#1}}
\newcommand{\eqnref}[1]{\eqref{eqn:#1}}
\newcommand{\twofigref}[2]{Figures~\ref{fig:#1} and \ref{fig:#2}}
\newcommand{\twosecref}[2]{Sections~\ref{sec:#1} and \ref{sec:#2}}
\newcommand{\tablabel}[1]{\label{tab:#1}}
\newcommand{\tabref}[1]{Table~\ref{tab:#1}}
\newcommand{\F}{\ensuremath{\mathcal{F}}}

\newcommand{\qn}[2][]{\ensuremath{\textup{\textsf{qn}}_{#1}(#2)}}
\newcommand{\sn}[2][]{\ensuremath{\textup{\textsf{sn}}_{#1}(#2)}}
\newcommand{\ctn}[1]{\textup{\ensuremath{\textsf{cat}(#1)}}}
\newcommand{\stack}[1]{\ensuremath{\textup{\textsf{stack}}(#1)}}
\newcommand{\queue}[1]{\ensuremath{\textup{\textsf{queue}}(#1)}}
\newcommand{\CEIL}[1]{\ensuremath{\protect\left\lceil#1\right\rceil}}
\newcommand{\HALF}{\ensuremath{\protect\dfrac{1}{2}}}
\newcommand{\quarter}{\ensuremath{\protect\tfrac{1}{4}}}
\newcommand{\bracket}[1]{\ensuremath{\protect\left(#1\right)}}
\newcommand{\aaa}{\textup{(a)}}
\newcommand{\bbb}{\textup{(b)}}
\newcommand{\ccc}{\textup{(c)}}

\newcommand{\Figure}[4][!htb]{
\begin{figure}[#1]
\begin{center}
#3
\caption{\figlabel{#2}#4}
\end{center}
\end{figure}}

\newcommand{\old}[1]{\textcolor{red}{OLD #1}}
\renewcommand{\old}[1]{}

\newcommand{\Oh}[1]{\ensuremath{\protect\mathcal{O}(#1)}}
\newcommand{\etal}{~et al.~}

\newcommand{\N}{\mathbb{N}}
\newcommand{\FLOOR}[1]{\ensuremath{\protect\left\lfloor#1\right\rfloor}}
\newcommand{\ceil}[1]{\lceil{#1}\rceil}
\newcommand{\floor}[1]{\lfloor{#1}\rfloor}
\newcommand{\half}{\ensuremath{\protect\tfrac{1}{2}}}
\renewcommand{\geq}{\geqslant}
\renewcommand{\leq}{\leqslant}

\DeclareMathOperator{\dist}{dist}

\begin{document}

\vspace*{2ex}
{\Large\bfseries\boldmath\scshape Thickness and Antithickness of Graphs}\footnotemark[1]

\DateFootnote

\medskip
\bigskip
{\large 
Vida Dujmovi{\'c}\,\footnotemark[2]
\quad
David~R.~Wood\,\footnotemark[4]
}

\bigskip
\bigskip
\emph{Abstract.} 
This paper studies questions about duality between crossings and non-crossings in graph drawings via the notions of thickness and antithickness. The \emph{thickness} of a graph $G$ is the minimum integer $k$ such that in some drawing of $G$, the edges can be partitioned into $k$ noncrossing subgraphs. The \emph{antithickness} of a graph $G$ is the minimum integer $k$ such that in some drawing of $G$, the edges can be partitioned into $k$ thrackles, where a \emph{thrackle} is a set of edges, each pair of which intersect exactly once. (Here edges with a common endvertex $v$ are considered to intersect at $v$.)\ So thickness is a measure of how close a graph is to being planar, whereas antithickness is a measure of how close a graph is to being a thrackle. This paper explores the relationship between the thickness and antithickness of a graph, under various graph drawing models, with an emphasis on extremal questions. 

%\author{Vida Dujmovi{\'c}\footnote{School of Computer Science,
%University of Ottawa, Ottawa, Canada
%    (\texttt{vida.dujmovic@uottawa.ca}).} \and David
%  R. Wood\footnote{School of Mathematical Sciences, Monash
%    University, Melbourne, Australia
%    (\texttt{david.wood@monash.edu}). }}
%    
%    %Research supported by the Australian Research Council.}}

\footnotetext[2]{School of Computer Science and Electrical Engineering,  
University of Ottawa, Ottawa, Canada (\texttt{vida.dujmovic@uottawa.ca}). Research  supported by NSERC and the Ontario Ministry of Research and Innovation.}

\footnotetext[4]{School of Mathematical Sciences, Monash   University, Melbourne, Australia  (\texttt{david.wood@monash.edu}). }
%Research supported by the Australian Research Council.}

%2000 MSC classification: 05C62 (graph representations),  05C10 (topological graph theory)}

%\newpage
%\tableofcontents
%\newpage

\renewcommand{\thefootnote}{\arabic{footnote}}
\setlength{\parskip}{2ex}

%%%%%%%%%%%%%%%%%%%%%%%%%%%%%%%%%%%%%%%%%%%%%%%%%%%%%%%%
\section{Introduction}\seclabel{Introduction}
%%%%%%%%%%%%%%%%%%%%%%%%%%%%%%%%%%%%%%%%%%%%%%%%%%%%%%%%

This paper studies questions about duality between crossings and
non-crossings in graph drawings. This idea is best illustrated by an
example. A graph is \emph{planar} if it has a drawing with no
crossings, while a \emph{thrackle} is a graph drawing in which every
pair of edges intersect exactly once. So in some sense, thrackles
are the graph drawings with the most crossings. Yet thracklehood and
planarity appear to be related. In particular, a widely believed
conjecture would imply that every thrackleable graph is planar. Loosely
speaking, this says that a graph that can be drawn with the maximum
number of crossings has another drawing with no crossings. This paper
explores this seemingly counterintuitive idea through the notions of
thickness and antithickness. First we introduce the
essential definitions.

A (\emph{topological}) \emph{drawing} of a graph\footnote{We consider
  undirected, finite, simple graphs $G$ with vertex set $V(G)$ and
  edge set $E(G)$. The number of vertices and edges of $G$ are
  respectively denoted by $n=|V(G)|$ and $m=|E(G)|$. Let $G[S]$ denote
  the subgraph of $G$ induced by a set of vertices $S\subseteq
  V(G)$. Let $G- S:=G[V(G)\setminus S]$ and $G- v:=G\setminus\{v\}$. 
%For all $A,B\subseteq V(G)$, let $G[A,B]$
 % denote the bipartite subgraph of $G$ with vertex set $A\cup B$ and
 % edge set $\{vw\in E(G):v\in A, w\in B\}$.
} $G$ is a function that
maps each vertex of $G$ to a distinct point in the plane, and maps
each edge of $G$ to a simple closed curve between the images of its
end-vertices, such that:
\begin{itemize}
\item the only vertex images that an edge image intersects are the
  images of its own end-vertices (that is, an edge does not `pass
  through' a vertex),
\item the images of two edges are not tangential at a common interior
  point (that is, edges cross `properly').
  % , and
  % \item the intersection of the images of any two edges is empty or
  %   consists of a single point (including an intersection at the
  %   image of a common end-vertex).
\end{itemize}
Where there is no confusion we henceforth do not distinguish between a
graph element and its image in a drawing. Two edges with a common
end-vertex are \emph{adjacent}. Two edges in a drawing \emph{cross} if
they intersect at some point other than a common end-vertex. Two edges
that do not intersect in a drawing are \emph{disjoint}. A drawing of a
graph is \emph{noncrossing} if no two edges cross. A graph is
\emph{planar} if it has a noncrossing drawing.

In the 1960s John Conway introduced the following definition. 
A drawing of a graph is a \emph{thrackle} if every pair of edges
intersect exactly once (either at a common endvertex or at a crossing point). A graph is \emph{thrackeable} if it has a drawing that is a
thrackle; see \citep{MN18,GX17,PP08,Cottingham93,RST16,CKN15,CN-DCG00,GR95,LPS-DCG97,PRS-DM94,CMN04,CK01,PS11,CN10,PRT12,FP11,AS17}. Note that in this definition, it is important that every pair of edges
intersect \emph{exactly} once since every graph has a drawing in which
every pair of edges intersect at least once\footnote{\emph{Proof}: Let
  $V(G)=\{v_1,\dots,v_n\}$. Position each vertex $v_i$ at
  $(i,0)$. Define a relation $\prec$ on $E(G)$ where $v_iv_j\prec
  v_pv_q$ if and only if $i<j\leq p<q$. Observe that $\preceq$ is a
  partial order of $E(G)$. Let $E(G)=\{e_1,\dots,e_m\}$, where
  $e_i\prec e_j$ implies that $j<i$. Draw each edge $e_i=v_pv_q$ as
  the 1-bend polyline $(p,0)(i,1)(q,0)$. Then every pair of edges
  intersect at least once.}.

A drawing is \emph{geometric} if every edge is a straight line
segment. A geometric drawing is \emph{convex} if every vertex is on
the convex hull of the set of vertices. A \emph{$2$-track drawing} is
a convex drawing of a bipartite graph in which the two colour classes
are separated in the ordering of the vertices around the convex
hull. For the purposes of this paper, we can assume that the two
colour classes in a $2$-track drawing are on two parallel lines
(called \emph{tracks}). The notion of a \emph{convex thrackle} is
closely related to that of \emph{outerplanar thrackle}, which was
independently introduced by \citet{CN10}.

%%%%%%%%%%%%%%%%%%%%%%%%%%%%%%%%%%%%%%%%%%%%%%%%%%%%%%%%
\subsection{Thickness and Antithickness}\seclabel{Thickness}
%%%%%%%%%%%%%%%%%%%%%%%%%%%%%%%%%%%%%%%%%%%%%%%%%%%%%%%%

The \emph{thickness} of a graph $G$ is the minimum $k\in\N$ such that
the edge set $E(G)$ can be partitioned into $k$ planar subgraphs. Thickness is
a widely studied parameter; see the surveys \citep{Hobbs69, MOS98}. 
The \emph{thickness} of a graph drawing is the minimum $k\in\N$ such that the edges of the drawing can
be partitioned into $k$ noncrossing subgraphs. Equivalently, each edge
is assigned one of $k$ colours such that crossing edges receive
distinct colours.

Every planar graph can be drawn with its vertices at prespecified
locations \citep{Kainen73,Halton91,PW-GC01}. It follows that a graph
has thickness $k$ if and only if it has a drawing with thickness $k$
\citep{Kainen73,Halton91}. However, in such a representation the edges
might be highly curved\footnote{In fact, \citet{PW-GC01} proved that
  for every planar graph $G$ that contains a matching of $n$ edges, if
  the vertices of $G$ are randomly assigned prespecified locations on
  a circle, then $\Omega(n)$ edges of $G$ have $\Omega(n)$ bends in
  every polyline drawing of $G$.}. The minimum integer $k$ such that a
graph $G$ has a geometric / convex / $2$-track drawing with thickness
$k$ is called the \emph{geometric} / \emph{book} / \emph{$2$-track thickness} of $G$. Book thickness is also called \emph{pagenumber} and \emph{stacknumber} in the literature; see the surveys \citep{Bilski-IE92,DujWoo04}\footnote{In the context of this paper it would make sense  to refer to book thickness as \emph{convex thickness}, and to refer to thickness as \emph{topological thickness}, although we refrain from the temptation of introducing further terminology. }

The following results
are well known for every graph $G$:
\begin{itemize}
\item $G$ has geometric thickness $1$ if and only if $G$ is planar
  \citep{Fary48, Wagner36}.
\item $G$ has book thickness $1$ if and only if $G$ is outerplanar
  \citep{Kainen73}.
\item $G$ has book thickness at most $2$ if and only if $G$ is a
  subgraph of a Hamiltonian planar graph \citep{Kainen73}.
\item $G$ has $2$-track thickness $1$ if and only if $G$ is a forest
  of caterpillars \citep{HS72}.
\end{itemize}

The \emph{antithickness} of a graph $G$ is the minimum $k\in\N$ such
that $E(G)$ can be partitioned into $k$ thrackeable subgraphs. The
\emph{antithickness} of a graph drawing is the minimum $k\in\N$ such
that the edges of the drawing can be partitioned into $k$
thrackles. Equivalently, each edge is assigned one of $k$ colours such
that disjoint edges receive distinct colours. The minimum $k\in\N$
such that a graph $G$ has a topological / geometric / convex / 2-track
drawing with antithickness $k$ is called the \emph{topological} /
\emph{geometric} / \emph{convex} / \emph{2-track antithickness} of
$G$. Thus a graph is thrackeable if and only if it has antithickness
$1$.

%%%%%%%%%%%%%%%%%%%%%%%%%%%%%%%%%%%%%%%%%%%%%%%%%%%%%%%%

\begin{lem}
  Every thrackeable graph $G$ has a thrackled drawing with each vertex
  at a prespecified location.
\end{lem}

\begin{proof}
  Consider a thrackled drawing of $G$. Replace each crossing point by a dummy
  vertex. Let $H$ be the planar graph obtained. Let $p(v)$ be a
  distinct prespecified point in the plane for each vertex $v$ of
  $G$. For each vertex $x\in V(H)-V(G)$ choose a distinct point
  $p(x)\in \mathbb{R}^2\setminus\{p(v):v\in V(G)\}$. Every planar graph can be drawn with its vertices at
  prespecified locations \citep{Kainen73,Halton91,PW-GC01}. Thus $H$
  can be drawn planar with each vertex $x$ of $H$ at $p(x)$. This
  drawing defines a thrackled drawing of $G$ with each vertex $v$ of
  $G$ at $p(v)$, as desired.
\end{proof}

\begin{cor}
\corlabel{AntithicknessCharacterisation}
  A graph has antithickness $k$ if and only if it has a drawing with
  antithickness $k$.
\end{cor}

%%%%%%%%%%%%%%%%%%%%%%%%%%%%%%%%%%%%%%%%%%%%%%%%%%%%%%%%%%%%%%%%%%%%%%%%%%%

Every graph $G$ satisfies
\begin{align*}
  & \text{thickness}(G) \;\leq\;\text{geometric thickness}(G)
  \;\leq\;\text{book thickness}(G)\enspace,\text{ and}\\
  & \text{antithickness}(G) \;\leq\;\text{geometric antithickness}(G)
  \;\leq\;\text{convex antithickness}(G)\enspace.
\end{align*}
Moreover, if $G$ is bipartite, then
\begin{align*}
  & \text{book thickness}(G)
  \;\leq\;
  \text{2-track thickness}(G)
  \;=\;
  \text{2-track antithickness}(G)
  \enspace,\text{ and }\\
  & \text{convex antithickness}(G)
  \;\leq\;
  \text{2-track thickness}(G)
  \;=\;
  \text{2-track antithickness}(G)
  \enspace.
\end{align*}
% \begin{align*}
%   & \text{book thickness}(G),\,
%   \text{convex antithickness}(G)\\
%   \;\leq\;&
%   \text{2-track thickness}(G)\\
%   \;=\;& \text{2-track antithickness}(G) \enspace.
% \end{align*}
For the final equality, observe that a $2$-track layout of $G$ with
antithickness $k$ is obtained from a $2$-track layout of $G$ with
thickness $k$ by simply reversing one track, and vice versa.

%%%%%%%%%%%%%%%%%%%%%%%%%%%%%%%%%%%%%%%%%%%%%%%%%%%%%%%%
\subsection{An Example: Trees}\seclabel{Trees}
%%%%%%%%%%%%%%%%%%%%%%%%%%%%%%%%%%%%%%%%%%%%%%%%%%%%%%%%

Consider the thickness of a tree. Every tree is planar, and thus has
thickness $1$ and geometric thickness $1$. It is well known that every
tree $T$ has $2$-track thickness at most $2$. \emph{Proof:} Orient the
edges away from some vertex $r$. Properly $2$-colour the vertices of
$T$ \emph{black} and \emph{white}. Place each colour class on its own
track, ordered according to a breadth-first search of $T$ starting at
$r$. Colour each edge according to whether it is oriented from a black
to a white vertex, or from a white to a black vertex. It is easily
seen that no two monochromatic edges cross, as illustrated in
\figref{TreeTrack}.

\Figure{TreeTrack}{\includegraphics{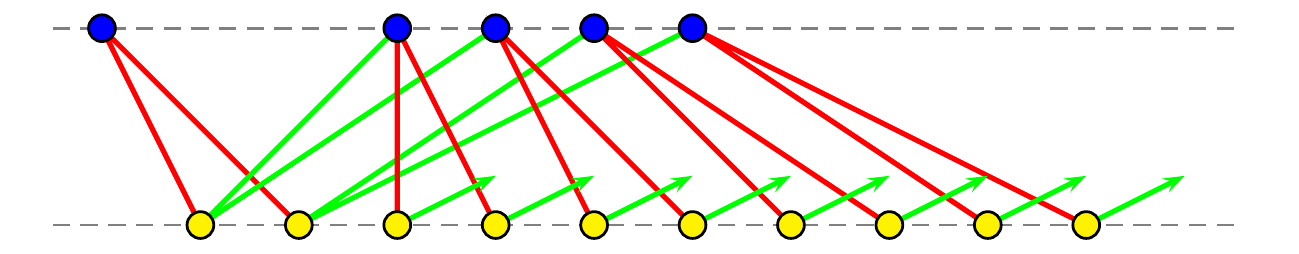}}{A $2$-track drawing of
  a tree with thickness $2$.}

The \emph{2-claw} is the tree with vertex set
$\{r,v_1,v_2,v_3,w_1,w_2,w_3\}$ and edge set
$\{rv_1,rv_2,rv_3,v_1w_1,v_2w_2,v_3w_3\}$, as illustrated in
\figref{TwoClaw}(a). The upper bound of $2$ on the $2$-track thickness
of trees is best possible since \citet{HS72} proved that the 2-claw
has $2$-track thickness exactly $2$, as illustrated in
\figref{TwoClaw}(b).

\Figure{TwoClaw}{\includegraphics{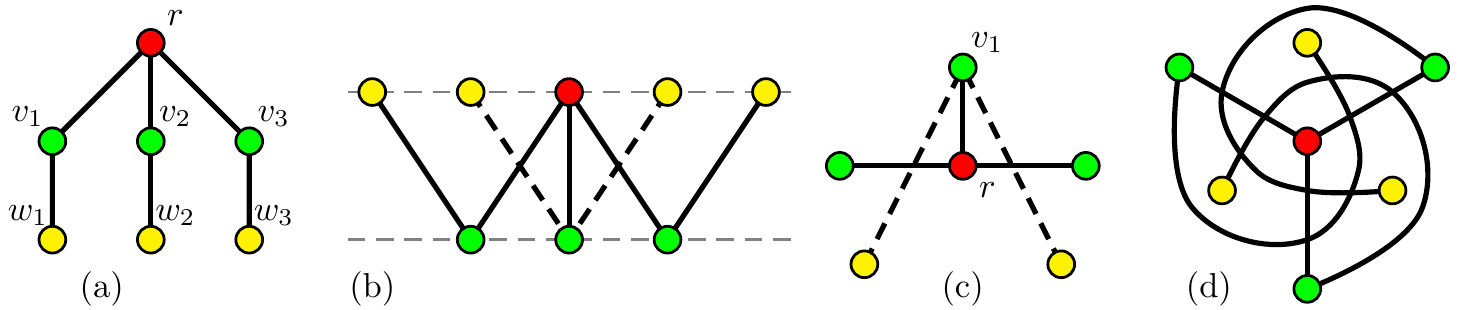}}{(a) The 2-claw.  (b) The
  2-claw has 2-track thickness 2. (c) The 2-claw is not a geometric
  thrackle. (d) The 2-claw drawn as a thrackle.}

What about the antithickness of a tree? Since every tree has $2$-track
thickness at most $2$, by reversing one track, every tree has
$2$-track antithickness at most $2$. And again the 2-claw shows that
this bound is tight.  In fact:

\begin{lem}
  \lemlabel{TwoClawGeometricThickness} The 2-claw is not a geometric
  thrackle.
\end{lem}

\begin{proof}
  Suppose to the contrary that the 2-claw is a geometric thrackle, as
  illustrated in \figref{TwoClaw}(c). For at least one of the three
  edges incident to $r$, say $rv_1$, the other two vertices adjacent
  to $r$ are on distinct sides of the line through $rv_1$. Thus
  $v_1w_1$ can only intersect one of $rv_2$ and $rv_3$, which is the
  desired contradiction.
\end{proof}

This lemma shows that $2$ is a tight upper bound on the geometric
antithickness of trees. On the other hand, if we allow curved edges,
\citet{Woodall-Thrackles} proved that every tree is thrackleable, and
thus has antithickness $1$, as illustrated in \figref{TwoClaw}(d) in
the case of a 2-claw.

%%%%%%%%%%%%%%%%%%%%%
\subsection{Main Results and Conjectures}\seclabel{Conjectures}
%%%%%%%%%%%%%%%%%%%%%%%%%%%%%%%%%%%%%%%%%%%%%%%%%%%%%%%%

A \emph{graph parameter} is a function $\beta$ that assigns to every
graph $G$ a non-negative integer $\beta(G)$. Examples that we have
seen already include thickness, geometric thickness, book thickness,
antithickness, geometric antithickness, and convex antithickness. Let
\F\ be a class of graphs.  Let $\beta(\F)$ denote the function
$f:\mathbb{N}\rightarrow\mathbb{N}$, where $f(n)$ is the maximum of
$\beta(G)$, taken over all $n$-vertex graphs $G\in\F$.  We say \F\ has
\emph{bounded} $\beta$ if $\beta(\F)\in\Oh{1}$ (where $n$ is the
hidden variable in $\Oh{1}$).

A graph parameter $\beta$ is \emph{bounded by} a graph parameter $\gamma$ (for some class \F), if there exists a \emph{binding} function $g$ such that $\beta(G)\leq g(\gamma(G))$ for every graph $G$ (in \F). If $\beta$ is bounded by $\gamma$ (in \F) and $\gamma$ is bounded by $\beta$ (in \F) then $\beta$ and $\gamma$ are \emph{tied} (in \F). If $\beta$ and $\gamma$ are tied, then a graph family \F\ has bounded $\beta$ if and only if \F\ has bounded $\gamma$. This definition is due to \citet{DO96} and \citet{Reed97}. Note that `tied' is a transitive relation. For $\beta$ and $\gamma$ to be not tied means that for some class of graphs, $\beta$ is bounded but $\gamma$ is unbounded (or vice versa). In this case, $\beta$ and $\gamma$ are \emph{separated}, which is terminology introduced by \citet{Eppstein01,Eppstein-AMS}.

The central questions of this paper ask which thickness/antithickness
parameters are tied. In \secref{Tied} we prove that thickness and
antithickness are tied---in fact we prove that these parameters are
both tied to arboricity, and thus only depend on the maximum density
of the graph's subgraphs. (See \secref{Tied} for the definition of arboricity.) 

\citet{Eppstein01} proved that book thickness and geometric thickness
are separated. In particular, for every $t$, there exists a graph with
geometric thickness $2$ and book thickness at least $t$; see
\citep{BO99,Blankenship-PhD03} for a similar result. Thus book
thickness is not bounded by geometric thickness. The example used here
is $K_n'$, which is the graph obtained from $K_n$ by subdividing each
edge exactly once. In \lemref{GeomAntiCompleteSubdiv} we prove that
$K_n'$ has geometric antithickness $2$. At the end of
\secref{Separating} we prove that $K_n'$ has convex antithickness at
least $\sqrt{n/6}$ (which is unbounded). Thus convex antithickness is
not bounded by geometric antithickness, implying that convex
antithickness and geometric antithickness are separated.

\citet{Eppstein-AMS} also proved that geometric thickness and
thickness are separated. In particular, for every $t$, there exists a
graph with thickness $3$ and geometric thickness at least $t$. Thus
geometric thickness is not bounded by thickness. (Note that it is open
whether every graph with thickness $2$ has bounded geometric
thickness.)\ \citet{Eppstein-AMS} used the following graph to
establish this result. Let $G_n$ be the graph having as its
$n+\binom{n}{3}$ vertices the singleton and tripleton subsets of an
$n$-element set, with an edge between two subsets when one is
contained in the other. (Note that $K_n'$ can be analogously defined---just replace tripleton by doubleton.)\ Then $G_n$ has thickness $3$, and for all $t$ there is an $n$ for which $G_n$ has geometric thickness at least $t$. We expect that an analogous separation result holds for antithickness and geometric antithickness.  Since $E(G_n)$ has an edge-partition into three star-forests, $G_n$ has antithickness $3$.  We conjecture that for all $t$ there is an $n$ for which $G_n$ has geometric antithickness at least $t$.  This would imply that geometric antithickness is not bounded by antithickness.

In the positive direction, we conjecture the following dualities:
%\begin{itemize}
%\item Geometric thickness and geometric antithickness are tied.
%\item Book thickness and convex antithickness are tied.
%\end{itemize}

\begin{conj}
\label{GeometricTied}
Geometric thickness and geometric antithickness are tied.
\end{conj}

\begin{conj}
\label{ConvexTied}
Book thickness and convex antithickness are tied.
\end{conj}

%\comment{DW: Are there any places in the paper, where we should refer to \cref{GeometricTied} or \cref{ConvexTied}?}

% \comment{Is there some sort of similar duality between (minimum)
%   crossing number and maximum crossing number? See \citep{Ringeisen,
%     GR92, PRS91, RSP-BICA91, SPR-JGT95} for results on maximum
%   crossing number.}

% \comment{Define the \emph{anti-crossing number} of a graph $G$ to be
%   the minimum number of pairs of disjoint edges in a drawing of
%   $G$. Is anti-crossing number tied to crossing number? Maybe not
%   for some planar graph? }

% I'm pretty sure that the standard probabilistic proof of the
% crossing lemma gives the same result for anti-crossing number, ie
% $\text{acr}(G)\geq cm^3/n^2$.

% The spanning subgraph of $G$ induced by a set of edges $F\subseteq
% E(G)$ is denoted by $G[F]$.

In \thmref{ThreeTied} we prove that convex antithickness and queue-number (defined in \secref{StackQueueTrackLayouts}) are tied. Thus the truth of \cref{ConvexTied} would imply that book thickness and queue-number are tied. This would imply, since planar graphs have bounded book thickness \citep{Yannakakis89,BS84}, that planar graphs have bounded queue-number, which is an open problem due to Heath\etal\citep{HR92,HLR92}; see \citep{DFP13,DMW17,Duj15} for recent progress. Thus a seemingly easier open problem is to decide whether planar graphs have bounded geometric antithickness.

\citet{LPS-DCG97} proved two related results. First they proved that
every bipartite thrackleable graph is planar. And more generally, they proved that a
bipartite graph has a drawing in which every pair of edges intersect
an odd number of times if and only if the graph is planar. In their
construction, non-adjacent edges cross once, and adjacent edges
intersect three times.
% \footnote{We include the simple proof for
%   completeness: Separate the colour classes by the strip $-1\leq y\leq
%   0$, and draw every edge straight in this strip. Then erase the
%   strip, and reflect the top-half of the drawing about the
%   Y-axis. Re-draw the edges within the strip straight. Every pair of
%   edges now cross once. Edges with a common end-vertex thus intersect
%   twice. Re-draw the local neighbourhood of each vertex, so that edges
%   with a common end-vertex intersect three times.}.

\subsection{Other Contributions}

In addition to the results discussed above, this paper makes the following contributions. 
In \secref{StackQueueTrackLayouts} we prove that convex antithickness is tied to queue-number and track-number. Several interesting results follow from this theorem. \secref{ColouringDrawing} surveys the literature on the problem of determining the thickness or antithickness of a given (uncoloured) drawing of a graph. 
\twosecref{Tied}{Separating} respectively prove two results discussed above, namely that thickness and antithickness are tied, and that convex antithickness and geometric antithickness are separated. \secref{Extremal} studies natural extremal questions for all of the above parameters. Finally, \secref{Complete} considers the various antithickness parameters for a complete graph. 

%%%%%%%%%%%%%%%%%%%%%%%%%%%%%%%%%%%%%%%%%%%%%%%%%%%%%%%%
\section{Stack, Queue and Track Layouts}\seclabel{StackQueueTrackLayouts}
%%%%%%%%%%%%%%%%%%%%%%%%%%%%%%%%%%%%%%%%%%%%%%%%%%%%%%%%

This section introduces track and queue layouts, which are well studied graph layout models. We show that they are closely related to convex antithickness.

A \emph{vertex ordering} of an $n$-vertex graph $G$ is a bijection
$\pi:V(G)\rightarrow\{1,2,\dots,n\}$. We write $v<_\pi w$ to mean that
$\pi(v)<\pi(w)$. Thus $\leq_\pi$ is a total order on $V(G)$. We say $G$
or $V(G)$ is \emph{ordered by} $<_\pi$. Let $L(e)$ and $R(e)$ denote
the end-vertices of each edge $e\in E(G)$ such that $L(e)<_\pi R(e)$. At
times, it will be convenient to express $\pi$ by the list
$(v_1,v_2,\dots,v_n)$, where $\pi(v_i)=i$. These notions extend to
subsets of vertices in the natural way. Suppose that
$V_1,V_2,\dots,V_k$ are disjoint sets of vertices, such that each
$V_i$ is ordered by $<_i$. Then $(V_1,V_2,\dots,V_k)$ denotes the
vertex ordering $\pi$ such that $v<_\pi w$ whenever $v\in V_i$ and
$w\in V_j$ with $i<j$, or $v\in V_i$, $w\in V_i$, and $v<_i w$. We
write $V_1<_\pi V_2<_\pi\dots<_\pi V_k$.

Let $\pi$ be a vertex ordering of a graph $G$. Consider two edges
$e,f\in E(G)$ with no common end-vertex. There are the following three
possibilities for the relative positions of the end-vertices of $e$ and
$f$ in $\pi$. Without loss of generality $L(e)<_\pi L(f)$.

\begin{itemize}

\item $e$ and $f$ \emph{cross}: $L(e)<_\pi L(f)<_\pi R(e)<_\pi R(f)$.

\item $e$ and $f$ \emph{nest} and $f$ is \emph{nested inside} $e$:
  $L(e)<_\pi L(f)<_\pi R(f)<_\pi R(e)$

\item $e$ and $f$ are \emph{disjoint}: $L(e)<_\pi R(e)<_\pi L(f)<_\pi
  R(f)$

\end{itemize}

A \emph{stack} (respectively, \emph{queue}) in $\pi$ is a set of edges
$F\subseteq E(G)$ such that no two edges in $F$ are crossing (nested)
in $\pi$. Observe that when traversing $\pi$, edges in a stack (queue)
appear in LIFO (FIFO) order---hence the names.

A \emph{linear layout} of a graph $G$ is a pair $(\pi,\{E_1,E_2,\dots,E_k\})$
where $\pi$ is a vertex ordering of $G$, and $\{E_1,E_2,\dots,E_k\}$
is a partition of $E(G)$. A $k$-\emph{stack} ($k$-\emph{queue})
\emph{layout} of $G$ is a linear layout $(\pi,\{E_1,E_2,\dots,E_k\})$
such that each $E_i$ is a \emph{stack} (\emph{queue}) in $\pi$.  At
times we write $\stack{e}=\ell$ (or $\queue{e}=\ell$) if $e\in
E_\ell$.

A graph admitting a $k$-stack ($k$-queue) layout is called a
$k$-\emph{stack} ($k$-\emph{queue}) \emph{graph}. The \emph{stack-number}
of a graph $G$, denoted by $\sn{G}$, is the minimum $k$ such that $G$
is a $k$-stack graph.  The \emph{queue-number} of $G$, denoted by
$\qn{G}$, is the minimum $k$ such that $G$ is a $k$-queue graph.
See \citep{DujWoo04} for a summary of results and references on
stack and queue layouts.

A $k$-stack layout of a graph $G$ defines a convex drawing of $G$ with thickness $k$, and vice versa. Thus the stack-number of $G$ equals the book thickness of $G$.

\begin{lem}
\lemlabel{QNCAT}
For every graph $G$, the queue-number of $G$ is at most the convex antithickness of $G$. 
\end{lem}

\begin{proof}
Consider a convex drawing of a graph $G$ with convex antithickness $k$. Let $(v_1,\dots,v_n)$ be the underlying circular ordering and let $E_1,\dots,E_k$ be the corresponding edge-partition. Any two edges in $E_i$ cross or intersect at a common end-vertex with respect to the vertex ordering $(v_1,\dots,v_n)$. Thus each $E_i$ is a queue, and $G$ has queue-number at most $k$. 
\end{proof}

% \comment{Note that convex thrackle layouts can be thought of as
%   non-nested non-disjoint linear layouts. There are two other
%   possible types of linear layouts that we have not considered:
%   \begin{itemize}
%   \item non-crossing non-nested
%   \item non-crossing non-disjoint (Here Dilworth's Theorem is
%     applicable to the fixed vertex ordering problem.)
%   \end{itemize}}

We now set out to prove a converse to \lemref{QNCAT}. A key tool will be track layouts, which generalise the notion of 2-track drawings, and have been previously studied by several authors
\citep{DMW05,DM-GD03,DPW04,DujWoo05,Miyauchi-Track,Miyauchi08,Miyauchi08a,DLMW-DM09,GLM-CGTA05,Duj15,DMW17}.

A \emph{vertex} $|I|$\emph{-colouring} of a graph $G$ is a partition
$\{V_i:i\in I\}$ of $V(G)$ such that for every edge $vw\in E(G)$, if
$v\in V_i$ and $w\in V_j$ then $i\ne j$. The elements of $I$ are
\emph{colours}, and each set $V_i$ is a \emph{colour class}. Suppose
that $<_i$ is a total order on each colour class $V_i$.  Then each
pair $(V_i,<_i)$ is a \emph{track}, and $\{(V_i,<_i):i\in I\}$ is an
$|I|$-\emph{track assignment} of $G$.  To ease the notation we denote
track assignments by $\{V_i:i\in I\}$ when the ordering on each colour
class is implicit.

An \emph{X-crossing} in a track assignment consists of two edges $vw$
and $xy$ such that $v<_ix$ and $y<_jw$, for distinct colours $i$ and
$j$.  An \emph{edge $k$-colouring} of $G$ is simply a partition
$\{E_i:1\leq i\leq k\}$ of $E(G)$.  A \emph{$(k,t)$-track layout} of
$G$ consists of a $t$-track assignment of $G$ and an edge
$k$-colouring of $G$ with no monochromatic X-crossing. A graph
admitting a $(k,t)$-track layout is called a \emph{$(k,t)$-track
  graph}. The \emph{track-number} of a graph $G$ is the minimum $t$
such that $G$ is a $(1,t)$-track graph.

% The minimum $t$ such that a graph $G$ is a $(k,t)$-track graph is
% denoted by \tn[k]{G}. The \emph{track-number} of a $G$ is
% $\tn{G}:=\tn[1]{G}$.

% \aaa\ Every $(k,2)$-track graph $G$ has convex antithickness $\ctn{G}\leq k$.\\

The next two lemmas give a method that constructs a convex drawing from a track layout.

\begin{lem}
  \lemlabel{ConvexAntithickness} Suppose that $K_t$ has a convex
  drawing with antithickness $p$, in which each thrackle is a matching. Then
  every $(k,t)$-track graph $G$ has convex antithickness at most $kp$.
\end{lem}

\begin{proof} 
In the given convex drawing of $K_t$, say the vertices are ordered $1,2,\dots,t$ around a
  circle, and $\{T_1,T_2,\dots,T_p\}$ is an edge-partition into thrackled matchings.  
  Let $\{(V_i,<_i):1\leq i\leq t\}$ be the track assignment and $\{E_\ell:1\leq\ell\leq k\}$ be the edge colouring in
  a $(k,t)$-track layout of $G$. Let
  $\pi=(V_1,V_2,\dots,V_t)$ be a circular vertex ordering of $G$.  For
  each $\ell\in\{1,2,\dots,k\}$ and $j\in\{1,2,\dots,p\}$, let
  $E_{\ell,j}=\{vw\in E_\ell:v\in V_{i_1},w\in V_{i_2},i_1i_2\in  T_j\}$.  
  We now show that each set $E_{\ell,j}$ is a convex thrackle in $\pi$, as illustrated in \figref{TrackConvexAntithickness}. 

\Figure{TrackConvexAntithickness}{\includegraphics{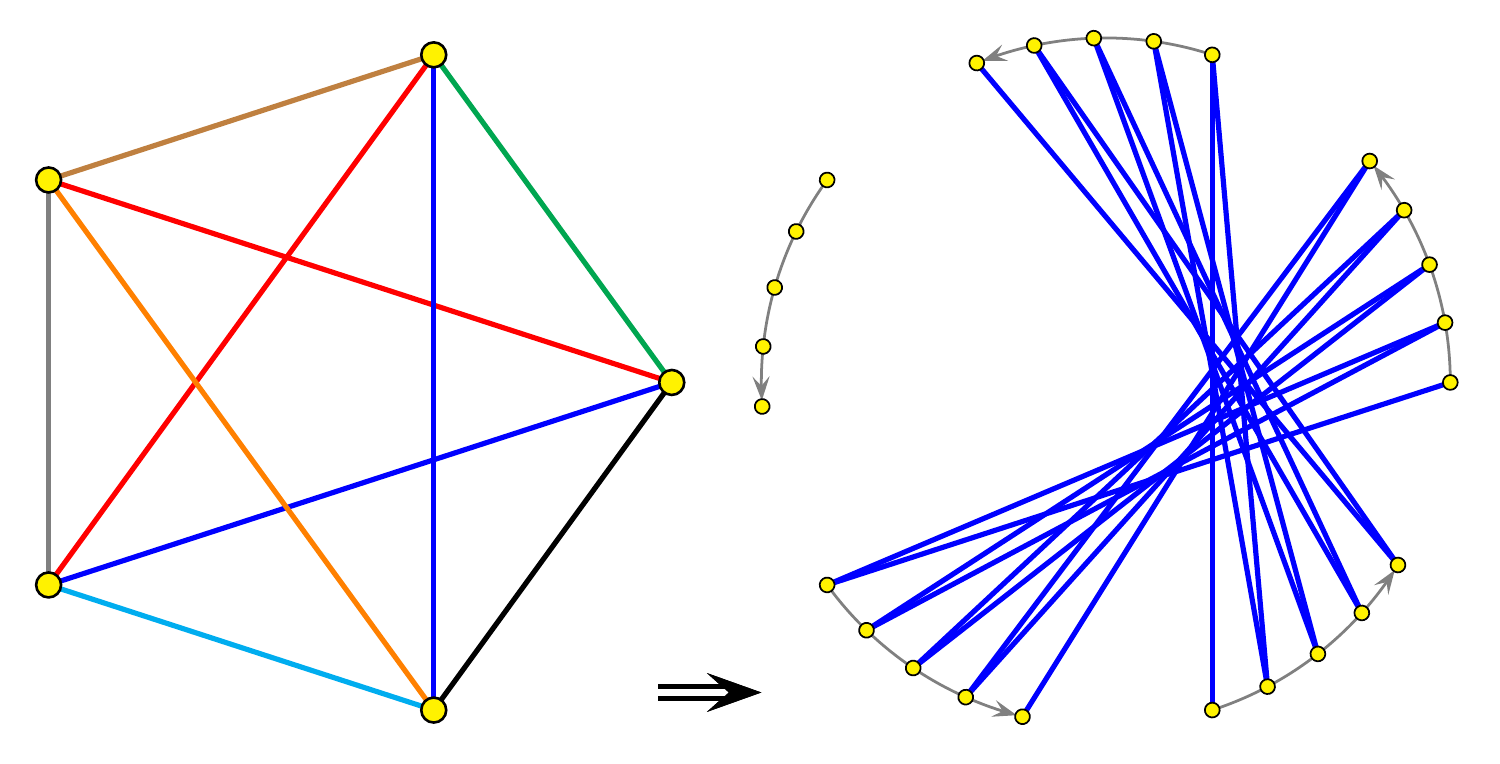}}{In the proof of \lemref{ConvexAntithickness}, starting from a 5-track graph $G$, and a convex drawing of $K_5$ with antithickness 8, in which each thrackle is a matching, replace the vertices of $K_5$ by the tracks to produce a convex drawing of $G$ with convex antithickness 8. 
}
  
  Consider two edges $e,f\in E_{\ell,j}$ with no common endvertex. If the endvertices of $e$ and $f$ belong to four distinct tracks, then $e$ and $f$ cross in $\pi$, since the edges in $T_j$ pairwise cross. The endvertices of $e$ and $f$ do not belong to three distinct tracks, since $T_j$ is a matching. If the endvertices of $e$ and $f$ belong to two distinct tracks, then $e$ and $f$ cross in $\pi$, as otherwise $e$ and $f$ form a monochromatic crossing in $G$. Thus  $E_{\ell,j}$ is a convex thrackle in $\pi$, and   $G$ has convex antithickness at most $kp$.
  \end{proof}

The following results show that $(k,t)$-track graphs have bounded convex antithickness (for bounded $k$ and $t$). We start by considering small values of $t$. 

\begin{lem}
  \lemlabel{Track}
  \aaa\ Every $(k,3)$-track graph $G$ has convex antithickness at most $3k$.\\
  \bbb\ Every $(k,4)$-track graph $G$ has convex antithickness at most $5k$.\\
  \ccc\ Every $(k,5)$-track graph $G$ has convex antithickness at most $8k$.
\end{lem}

\begin{proof}
By  \lemref{ConvexAntithickness}, it is enough to show that $K_3$, $K_4$ and $K_5$ admit convex drawings with their edges partitioned into 3, 5, and 8 thrackled matchings, respectively. The first claim is trivial. For the second claim, position the vertices of $K_4$ around a circle. Colour the two crossing edges blue. Colour each of the four other edges by a
  distinct colour. We obtain a convex drawing of $K_4$ with its edges
  partitioned into five thrackled matchings. Finally, say $V(K_5)=\{1,2,3,4,5\}$. Position the vertices of $K_5$ around a
  circle in the order $1,2,3,4,5$. Then
  $\{13,24\},\{25,14\},\{35\},\{12\},\{23\},\{34\},\{45\},\{15\}$ is a
  partition of $E(K_5)$ into 8 thrackled matchings. 
\end{proof}

% It is easily seen that every tree has $2$-track thickness at most
% $2$.

\citet{DPW04} proved that every outerplanar graph has a
$5$-track layout. Thus \lemref{Track}(c) with $k=1$ implies:

\begin{cor}
  Every outerplanar graph $G$ has convex antithickness at most $8$.
\end{cor}

%%%%%%%%%%%%%%%%%%%%%%%%%%%%%%%%%%%%%%%%%%%%%%%%%%%%%%%%%%

Let $\ln n$ be the natural logarithm of $n$. 
  Let $H(n):=\sum_{i=1}^n\frac{1}{i}$ denote the $n$-th harmonic
number. It is well-known that
\begin{equation}
  \eqnlabel{Harmonics}
  \ln n+\gamma
  \leq 
  H(n)
  <
  \ln n+\gamma+\frac{1}{2n},
%  <
%  1+\ln n,
\end{equation}
where $\gamma=0.577\ldots$ is the Euler--Mascheroni constant; see \citep{HarmonicNumber,GKP94}.

The constructions in the proof of \lemref{Track} generalise as follows. 

\begin{lem}
\lemlabel{CompleteMatching} 
The complete graph $K_n$ has a convex drawing with antithickness $p$ in which every thrackle is a matching, for some integer $p<n \ln(2n)$. That is, there is an edge $p$-colouring, such that edges that are disjoint or have a vertex in common receive distinct  colours.
\end{lem}

\begin{proof} 
By the above constructions, we may assume that $n\geq 6$. 
  Let $(v_0,v_1,\dots,v_{n-1})$ be the vertices of $K_n$ in order
  around a circle. For each $\ell\in\{1,2,\dots,\floor{\frac{n}{2}}\}$
  and $j\in\{0,1,\dots,\ceil{n/\ell}-1\}$, let $E_{\ell,j}$ be the set
  of edges $$E_{\ell,j}:=\{v_iv_{(i+\ell)\bmod{n}}:j\ell\leq
  i\leq(j+1)\ell-1\}.$$ 
 As illustrated in \figref{ConvexMatchingThrackle},    $E_{\ell,j}$ is a thrackle and a  matching.
 
 \Figure{ConvexMatchingThrackle}{\includegraphics{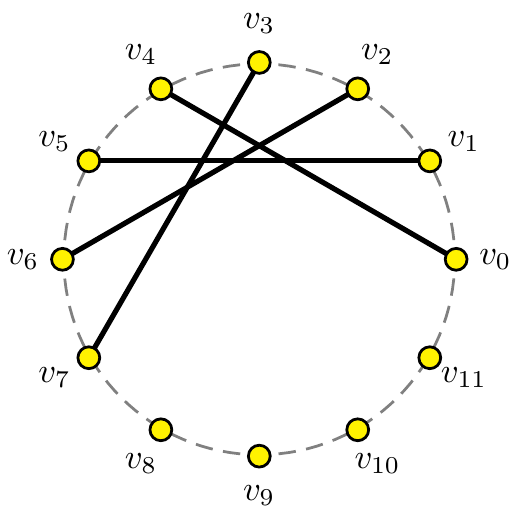}}{The
  set of edges $E_{\ell,j}$ in \lemref{CompleteMatching} with $\ell=4$
  and $j=0$.}

  Now
  \begin{align*}
    p
     \;=\;
    \sum_{\ell=1}^{\floor{\frac{n}{2}}}\ceil{n/\ell}
    \;\leq\;
    \sum_{\ell=1}^{\floor{\frac{n}{2}}}\frac{n+\ell-1}{\ell}
   & \;=\;
    (n-1)\cdot H(\floor{\tfrac{n}{2}})+\floor{\tfrac{n}{2}}.
  \end{align*}
By     \eqnref{Harmonics}, 
  \begin{align*}
p 
\;&<\; 
(n-1) \left( \ln \floor{\tfrac{n}{2}} + \gamma+\frac{1}{2\floor{\tfrac{n}{2}}} \right) +\floor{\tfrac{n}{2}}   \\
\;&\leq \; 
n\ln(n) - n \ln(2) + \gamma n + \frac{n-1}{2\floor{\tfrac{n}{2}}}  +\tfrac{n}{2}   \\
\;&\leq \; 
n\ln(n)  + n \left( - \ln(2) + \gamma + \tfrac{1}{2} \right)  + \tfrac{1}{4}   \\
\;&< \; 
n\ln(n)  + \frac{2n}{5}  + \tfrac{1}{4}   \\
 \;&<\;
n\ln(n)  + \frac{n}{2}  \\
 \;&<\;
    n\ln(2n).\quad \qedhere
  \end{align*}
\end{proof}

%WA: (n-1)(log(n/2) + 0.577 + 1/(2 floor(n/2) )) + floor(n/2) <= n log(2n)

%%%%%%%%%%%%%%%%%%%%%%%%%%%%%%%%%%%%%%%%%%%%%%%%%%%%%%%%%%%%%%%%%%%%%%%

\twolemref{ConvexAntithickness}{CompleteMatching} imply:

\begin{thm}
  \thmlabel{kt-Track} Every $(k,t)$-track graph $G$ has convex
  antithickness at most $kt\ln(2t)$.
\end{thm}

% \begin{proof} Let $\{(V_i,<_i):1\leq i\leq t\}$ be the track
%   assignment and $\{E_\ell:1\leq\ell\leq k\}$ be the edge colouring
%   in a $(k,t)$-track layout of $G$. By \lemref{CompleteMatching},
%   $K_t$ has a convex thrackle layout $\{T_1,T_2,\dots,T_{(t+1)(1+\ln
%     t)}\}$ in which each $T_i$ is a matching. Let
%   $\pi=(V_1,V_2,\dots,V_t)$ be a circular vertex ordering of
%   $G$. Let $E_{\ell,j}=\{vw\in E(G):v\in V_{i_1},w\in
%   V_{i_2},i_1i_2\in T_j\}$.  Clearly each $E_{\ell,j}$ is a convex
%   thrackle in $\pi$. \end{proof}

A similar result was proved by \citet{DPW04}, who showed that a
$(k,t)$-track graph has geometric thickness at most
$k\ceil{\frac{t}{2}}\floor{\frac{t}{2}}$. It is interesting that track
layouts can be used to produce graph drawings with small geometric
thickness, and can be used to produce graph drawings with small convex
antithickness. 

\citet{DPW04} proved that every $k$-queue $c$-colourable graph
has a $(2k,c)$-track layout. Thus \thmref{kt-Track} implies:

\begin{cor}
  \corlabel{QueueColour2ConvexAntithickness} Every $k$-queue
  $c$-colourable graph $G$ has convex
  antithickness at most $2kc\ln(2c)$. 
  %\leq 2k(4k+1)(1+\ln  4k).$$
\end{cor}

\citet{DujWoo04} proved that every
$k$-queue graph $G$ is $4k$-colourable, and thus has a $(2k,4k)$-track
layout. Thus \thmref{kt-Track} implies:

\begin{cor}
  \corlabel{Queue} Every $k$-queue graph $G$ has convex antithickness at most $8k^2\ln(8k)$. 
\end{cor}

A graph is \emph{series parallel} if it has no $K_4$-minor.

\begin{thm}
  Every series parallel graph $G$ has convex antithickness at most $18$.
\end{thm}

\begin{proof} 
  It is well known that $G$ is 3-colourable. \citet{RM-COCOON95}  proved that $G$ has a $3$-queue layout; see \citep{DujWoo-TR-02-03} for an  alternative proof. By the  above-mentioned result of \citet{DPW04}, $G$ has a  $(6,3)$-track layout.  The result follows from \lemref{Track}(a).
\end{proof}

\citet{DPW04} proved that queue-number and track-number are tied. 
Thus \lemref{QNCAT} and  \corref{Queue} imply:

\begin{thm}
  \thmlabel{ThreeTied} 
  Queue-number, track-number and convex antithickness are tied.
\end{thm}

Several upper bounds on convex antithickness immediately follow from known upper bounds on queue-number. In particular,  \citet{DMW05} proved that graphs with bounded treewidth have bounded track-number (see \citep{GLM-CGTA05,Wiechert17} for quantitative improvements to the bounds). Thus \thmref{kt-Track} with $k=1$ implies that graphs with bounded treewidth have bounded  convex antithickness. Improving on a breakthrough by \citet{DFP13}, \citet{Duj15}  proved that $n$-vertex planar graphs have $O(\log n)$ queue-number. More generally, \citet{DMW17} proved that $n$-vertex graphs with Euler genus $g$ have $O(g + \log n)$ queue-number. Since such graphs are $O(\sqrt{g})$-colourable, \corref{QueueColour2ConvexAntithickness} implies a $O(\sqrt{g}(\log g)(g+ \log n))$ upper bound on the convex antithickness. Most generally, for fixed $H$, \citet{DMW17} proved that $H$-minor-free graphs have $\log^{O(1)}n$ queue-number, which implies a $\log^{O(1)}n$ bound on the convex antithickness (since such graphs are $O(1)$-colourable).

%%%%%%%%%%%%%%%%%%%%%%%%%%%%%%%%%%%%%%%%%%%%%%%%%%%%%%%%
\section{Thickness and Antithickness of a Drawing}\seclabel{ColouringDrawing}
%%%%%%%%%%%%%%%%%%%%%%%%%%%%%%%%%%%%%%%%%%%%%%%%%%%%%%%%

This section considers the problem of determining the thickness or
antithickness of a given drawing of a graph. We employ the following
standard terminology. For a graph $G$, a \emph{clique} of $G$ is a set
of pairwise adjacent vertices of $G$.  The \emph{clique number} of
$G$, denoted by $\omega(G)$, is the maximum number of vertices in a
clique of $G$.  The \emph{clique-covering number} of $G$, denoted by
$\sigma(G)$ is the minimum number of cliques that partition $V(G)$.
An \emph{independent set} of $G$ is a set of pairwise nonadjacent
vertices of $G$.  The \emph{independence number} of $G$, denoted by
$\alpha(G)$, is the maximum number of vertices in an independent set of
$G$.  The \emph{chromatic number} of $G$, denoted by $\chi(G)$, is the
minimum number of independent sets that partition $V(G)$.  Obviously
$\chi(G)\geq\omega(G)$ and $\sigma(G)\geq\alpha(G)$ for every graph
$G$. Let \F\ be a family of graphs. \F\ is \emph{$\chi$-bounded} if
$\chi$ is bounded by $\omega$ in \F, and \F\ is
\emph{$\sigma$-bounded} if $\sigma$ is bounded by $\alpha$ in \F.

Now let $D$ be a drawing of a graph $G$. Let $k$ be the maximum number
of pairwise crossing edges in $D$, and let $\ell$ be the maximum
number of pairwise disjoint edges in $D$. Then $k$ is a lower bound on
the thickness of $D$, and $\ell$ is a lower bound on the antithickness
of $D$. Our interest is when the thickness of $D$ is bounded from
above by a function of $k$, or the antithickness of $D$ is bounded
from above by a function of $\ell$.

Let $H$ be the graph with $V(H)=E(G)$ such that two vertices of $H$ are adjacent if and only if the corresponding edges cross in $D$.  Let $H^+$ be the graph  with $V(H)=E(G)$ such that two vertices of $H^+$ are adjacent if and only if the corresponding edges cross in $D$ or have an endvertex in common. Note that $H$ is a spanning subgraph of $H^+$.  By definition, the thickness of $D$ equals $\chi(H)$, and the antithickness of $D$ equals $\sigma(H^+)$. 

A \emph{string graph} is the intersection graph of a family of simple curves in the plane; see
\citep{JU17, FP12,FP14,FP10,Mat14,SS-JCSS04} for example. 
% \citep{, PachToth-DCG02,KGK86, , MP-DM93, MP-DM92, Krat-JCTB91, SSS-JCSS03}. 
  If we consider edges in $D$ as curves, then $H^+$ is a string graph. 
  And deleting a small disc around each vertex in $D$, we see that $H$ is also a string graph.  
 Moreover, if $D$ is geometric, then both $H$ and $H^+$ are intersection graphs of sets of segments in the
plane. If $D$ is convex, then both $H$ and $H^+$ are intersection graphs of sets
of chords of a circle, which is called a \emph{circle graph}; see  \citep{GL-DM85,Kostochka04} for example. 
%\citep{Ageev-DM96, Ageev-DM99, Kostochka88, Kostochka04, Cerny07, KK-DM97, GL-DM85}. 
If $D$ is a $2$-track drawing, then both $H$ and $H^+$ are permutation graphs, which are perfect; see \citep{Golumbic80} for example.

Whether the thickness / antithickness of an (unrestricted) drawing is bounded by the maximum number of pairwise crossing / disjoint edges is equivalent to whether string graphs are $\chi$-bounded / $\omega$-bounded. Whether the thickness / antithickness of a geometric drawing is bounded by the maximum number of pairwise crossing / disjoint edges is equivalent to whether intersection graphs of segments are $\chi$-bounded / $\omega$-bounded. For many years both these were open; see \citep{KN-EJC02, KN-EJC98}. However, in a recent breakthrough, \citet{Pawlik13,Pawlik14} constructed set of segments in the plane, whose intersection graph is triangle-free and with unbounded chromatic number. Thus the thickness of a drawing is not bounded by any function of the maximum number of pairwise crossing edges, and this remains true in the geometric setting. 

For convex drawings, more positive results are known. \citet{Gyarfas-DM85,Gyarfas-DM86} proved that the family of circle graphs is $\chi$-bounded, the best known bound being $\chi(H)<21\cdot2^{\omega(H)}$ due to \citet{Cerny07} (slightly improving an earlier bound by \citet{KK-DM97}). This implies that a convex drawing with at most $k$ pairwise crossing edges has thickness
less than $21\cdot 2^{k}$. For small values of $k$ much better bounds are known \citep{Ageev-DM96, Ageev-DM99}. \citet{Kostochka88} proved that $\sigma(H)\leq(1+o(1))\alpha(H)\log\alpha(H)$ for every circle graph $H$; also see \citep{KK-DM97}. Thus a convex drawing with at most $k$ pairwise disjoint edges has antithickness at most $(1+o(1))k\log k$.

Now consider a 2-track drawing $D$. Then $H$ and $H^+$ are permutation graphs, which are perfect. Thus $\chi(H)=\omega(H)$ and $\sigma(H^+)=\alpha(H^+)$. This says that if $D$ has at most $k$ pairwise crossing edges, and at most $\ell$ pairwise disjoint edges, then $D$ has thickness at most $k$ and antithickness at most $\ell$. There is a very simple algorithm for computing these partitions. First we compute the partition into $\ell$ 2-track thrackles. For each edge $vw$, if $\{x_1y_1,x_2y_2,\dots,x_iy_i\}$ is a set of maximum size of pairwise disjoint edges such that $x_1<x_2<\dots<x_i<v$ in one layer and $y_1<y_2<\dots<y_i<w$ in the other layer, then assign $vw$ to the ($i+1$)-th set. Consider two disjoint edges $v_1w_1$ and $v_2w_2$. Without loss of generality, $v_1<v_2$ and $w_1<w_2$. Suppose that by the above rule $v_1w_1$ is assigned to the ($i+1$)-th set and $v_2w_2$ is assigned to the ($j+1)$-th set. Let $\{x_1y_1,\dots,x_iy_i\}$ be a set of maximum size of pairwise disjoint edges such
that $x_1<x_2<\dots<x_i<v_1$ in one layer and $y_1<y_2<\dots<y_i<w_1$ in the other layer. Then $\{x_1y_1,\dots,x_iy_i,v_1w_1\}$ 
is a set of pairwise disjoint edges such that $x_1<x_2<\dots<x_i<v_1<v_2$ in one layer and $y_1<y_2<\dots<y_i<w_1<w_2$ in the other layer. Thus $j\geq i+1$. That is, two edges that are both assigned to the same set are not disjoint, and each such set is a 2-track thrackle. In the above rule, $i\leq\ell-1$. Thus this procedure partitions the edges into $\ell$ 2-track thrackles. To partition the
edges into $k$ 2-track noncrossing subdrawings, simply reverse the order of the vertices in one track, and apply the above procedure.

Consider the analogous question for queue layouts: Given a fixed vertex ordering $\pi$ of a graph $G$, determine the minimum value $k$ such that $\pi$ admits a $k$-queue layout of $G$. We can again construct an auxillary graph $H$ with $V(H)=E(G)$, where two vertices are adjacent if and only if the corresponding edges of $G$ are nested in $\pi$. Then $\pi$ admits a $k$-queue layout of $G$ if and only if $\chi(H)\leq k$. A classical result by \citet{DM41} implies that $H$ is a permutation graph, and is thus perfect. Hence $\chi(H)=\omega(H)$. A clique in $H$ corresponds to a set of edges of $G$ that are pairwise nested in $\pi$, called a \emph{rainbow}. Hence $\pi$ admits a $k$-queue layout of $G$ if and only if $\pi$ has no ($k+1$)-edge rainbow, which was also proved by \citet{HR92}.  \citet{DujWoo04} observed the following simple way to assign edges to queues: if the maximum number of edges that are pairwise nested inside an edge $e$ is $i$, then assign $e$ to the ($i+1$)-th queue.

This procedure can also be used to prove that a convex drawing with at
most $k$ pairwise disjoint edges has antithickness at most
$(1+o(1))k\log k$. This is equivalent to the result of \citet{Kostochka88} mentioned
above. Let $\pi$ be any vertex ordering obtained from the order of the
vertices around the convex hull. Then $\pi$ has no ($k+1$)-edge
rainbow. Assign edges to $k$ queues as described at the end of the previous paragraph. Partition the
$i$-th queue into sets of pairwise non-disjoint edges as follows. For
each edge $e$ in the $i$-th queue, if the maximum number of pairwise
disjoint edges with both end-vertices to the left of the left end-vertex of $e$ is $j$, then assign $e$ to the ($j+1$)-th set. Thus
two edges in the $i$-th queue that are both assigned to the same set
are not disjoint. Let $S$ be a maximum set of pairwise disjoint edges
in the $i$-th queue. Then $j\leq|S|-1$. Thus the $i$-th queue can
be partitioned into $|S|$ sets of pairwise non-disjoint edges. Now we
bound $|S|$. Under each edge in $S$ is an $(i-1)$-edge rainbow. This
gives a set of $|S|\cdot i$ edges that are pairwise disjoint. Thus
$|S|\cdot i\leq k$ and $|S|\leq\floor{k/i}$. Thus we can partition the
$i$-th queue into at most $\floor{k/i}$ sets of pairwise non-disjoint
edges. In total we have at most $\sum_{i=1}^k\floor{k/i}$ sets, each
with no two disjoint edges, which is less than $k(1+\ln k)$. Loosely
speaking, this proof shows that a convex drawing and an associated
edge-partition into convex thrackles can be
thought of as a combination of a queue layout and an arch layout; see
\citep{DujWoo04} for the definition of an arch layout.

% \comment{With a bit more work we get $\ceil{\frac{k}{2}}+k\ln
%   k$. Can we claim to improve the \citet{Kostochka88} result? }

%One can think of a stack layout has having a circular vertex ordering with no crossing  
%edges in each stack. In a circular vertex ordering, two edges with no end-vertex in common 
%are either crossing or \emph{parallel}.  A set of edges $S$ is a \emph{convex thrackle} 
%with respect to a circular vertex ordering $\sigma$ if no two edges in $S$ are parallel. A 
%\emph{convex $k$-thrackle layout} of a graph $G$ consists of a vertex ordering $\sigma$ of 
%$G$, and a partition $\{E_1,E_2,\dots,E_k\}$ of $E(G)$ such that each $E_i$ is a convex 
%thrackle in $\sigma$. We say $G$ is a \emph{convex $k$-thrackle graph}. The \emph{convex 
%thrackle number} of a graph $G$, denoted by \ctn{G}, is the minimum $k$ such that $G$ is a 
%convex $k$-thrackle graph.

%%%%%%%%%%%%%%%%%%%%%%%%%%%%%%%%%%%%%%%%%%%%%%%%%%%%%%%%
\section{Thickness and Antithickness are Tied}\seclabel{Tied}
%%%%%%%%%%%%%%%%%%%%%%%%%%%%%%%%%%%%%%%%%%%%%%%%%%%%%%%%

The \emph{arboricity} of a graph $G$ is the minimum number of forests
that partition $E(G)$. \citet{NW-JLMS64} proved that the arboricity of
$G$ equals
\begin{equation}
  \eqnlabel{Arboricity}
  \max_{H\subseteq G}\CEIL{\frac{|E(H)|}{|V(H)|-1}}\enspace.
\end{equation}

We have the following connection between thickness, antithickness, and
arboricity.

\begin{thm}
  \thmlabel{ThicknessTied} Thickness, antithickness, and arboricity
  are pairwise tied. In particular, for every graph $G$ with thickness
  $t$, antithickness $k$, and arboricity $\ell$,
$$k\leq\ell\text{ and }\frac{k}{3}\leq t\leq\ell\leq\CEIL{\frac{3k}{2}}.$$
\end{thm}

\begin{proof}
  Every forest is planar. Thus a partition of $G$ into $\ell$ forests
  is also a partition of $G$ into $\ell$ planar subgraphs. Thus
  $t\leq\ell$.

  \citet{Woodall-Thrackles} proved that every forest is
  thrackeable. Thus a partition of $G$ into $\ell$ forests is also a
  partition of $G$ into $\ell$ thrackeable subgraphs. Thus
  $k\leq\ell$.

  Every planar graph $G$ has arboricity at most $3$ by
  \eqnref{Arboricity} and since $|E(G)|\leq3|V(G)|-6$.  (Indeed, much more is known about edge-partitions of planar graphs into three forests~\citep{Schnyder-Order89}.)\   Since every  forest is thrackeable \citep{Woodall-Thrackles}, every planar graph
  has antithickness at most $3$.  Thus a partition of $G$ into $t$
  planar subgraphs gives a partition of $G$ into $3t$ thrackeable
  subgraphs.  Thus $k\leq3t$.

  It remains to prove that $\ell\leq\CEIL{\frac{3}{2}k}$. By
  \eqnref{Arboricity}, it suffices to show that
  $\frac{m}{n-1}\leq\frac{3k}{2}$ for every subgraph $H$ of $G$ with
  $n$ vertices and $m$ edges.  \citet{CN-DCG00} proved that every
  thrackle has at most $\frac{3}{2}(n-1)$ edges. Since every subgraph
  of a thrackle is a thrackle, $H$ has antithickness at most $k$, and
  $m\leq\frac{3}{2}k(n-1)$, as desired.
\end{proof}

\thmref{ThicknessTied} with $k=1$ implies that every thrackle has arboricity at most $2$. It is an open problem whether every thrackle is planar. It follows from a result by \citet{CN09} that every thrackle has a crossing-free embedding in the projective plane. Also note that the constant $\frac{3}{2}$ in \thmref{ThicknessTied} can be improved to $\frac{167}{117}$ using the result of \citet{FP11}. 

%%%%%%%%%%%%%%%%%%%%%%%%%%%%%%%%%%%
\section{Separating Convex Antithickness and Geometric Thickness}\seclabel{Separating}
%%%%%%%%%%%%%%%%%%%%%%%%%%%%%%%%%%%

As discussed in \secref{Conjectures}, the following lemma is a key step in showing that convex antithickness and geometric antithickness are separated. Recall that $K_n'$ is the graph obtained from $K_n$ by subdividing each edge exactly once. 

\begin{lem}
\lemlabel{GeomAntiCompleteSubdiv}
$K_n'$ has geometric antithickness $2$.
\end{lem}

\begin{proof}
Let $v_1,\dots,v_n$ be the original vertices of $K_n'$.
Position each $v_i$ at $(2i,0)$.
For $1\leq i<j\leq n$, 
let $x_{i,j}$ be the division vertex of the edge $v_iv_j$; 
colour the edge $v_ix_{i,j}$ \emph{blue}, and
colour the edge $v_jx_{i,j}$ \emph{red}.
Orient each edge of $K_n'$ from the original endvertex to the division endvertex.
This orientation enables us to speak of the order of crossings along an edge.

We now construct a geometric drawing of $K_n'$, such that every pair of blue edges crosses, and every pair of red edges cross. Thus the drawing has antithickness $2$. In addition, the following invariants are maintained for all $i\in[1,n-2]$ and $j\in [i+2,n]$:

(1) No blue edge crosses $v_ix_{i,i+1}$ after the crossing between $v_ix_{i,i+1}$ and $v_jx_{j-1,j}$.

(2) No red edge crosses $v_jx_{j-1,j}$ after the crossing between $v_ix_{i,i+1}$ and $v_jx_{j-1,j}$.

% (1) For each blue edge $v_ix_{i,j}$, no blue edge crosses $v_ix_{i,j}$ after a red edge crosses $v_ix_{i,j}$.

% (2) For each red edge $v_jx_{i,j}$, no red edge crosses $v_jx_{i,j}$ after a blue edge crosses $v_jx_{i,j}$.\\

\noindent The drawing is constructed in two stages. First, for $i\in[n-1]$, position $x_{i,i+1}$ at $(2(n-i)+1,1)$, as illustrated in \figref{CompleteSubdiv}.

\Figure{CompleteSubdiv}{\includegraphics{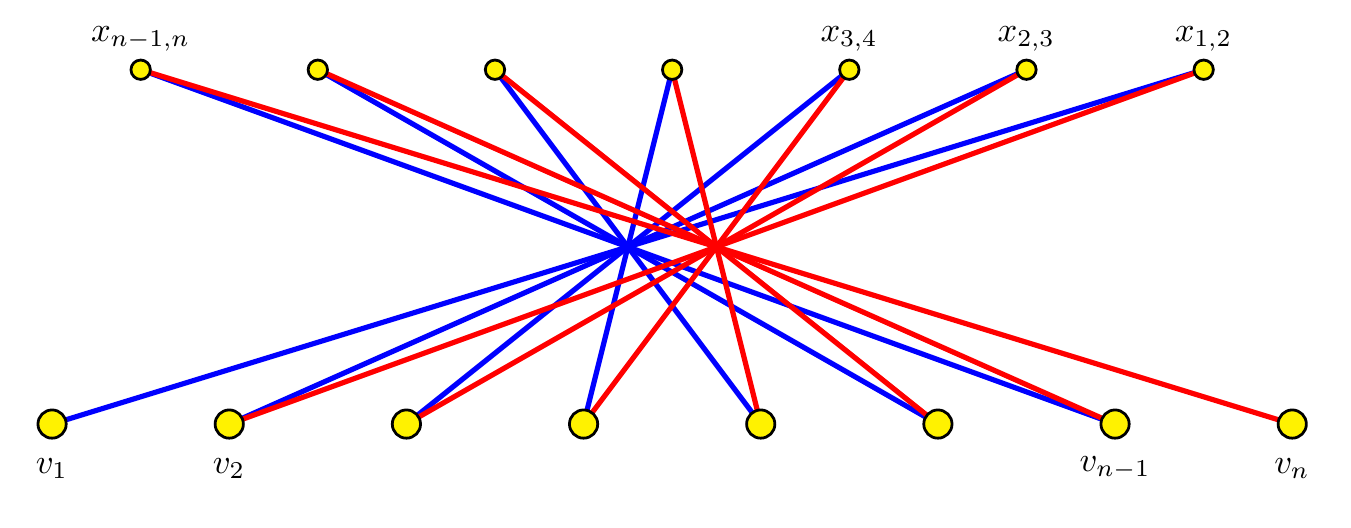}\vspace*{-4ex}}{Initial vertex placement in a geometric drawing of $K_n'$ with antithickness $2$.}

Observe that all the blue segments intersect at $(n+\half,\half)$, and
all the red segments intersect at $(n+\frac32,\half)$. Thus the invariants hold in this subdrawing. Moreover, every blue edge intersects every red edge (although this property will not be maintained).

For $i\in[1,n-2]$ and $j\in[i+2,n]$ (in an arbitrary order) position $x_{i,j}$ as follows. The blue segment $v_ix_{i,i+1}$ and the red segment $v_jx_{j-1,j}$ were drawn in the first stage, and thus cross at some point $c$. In the arrangement formed by the drawing produced so far, let $F$ be the face that contains $c$, such that the blue segment $v_ix_{i,i+1}$ is on the left of $F$, and the red segment $v_jx_{j-1,j}$ is on the right of $F$. Position $x_{i,j}$ in the interior of $F$, as illustrated in \figref{CompleteSubdivDetail}.

\Figure{CompleteSubdivDetail}{\includegraphics{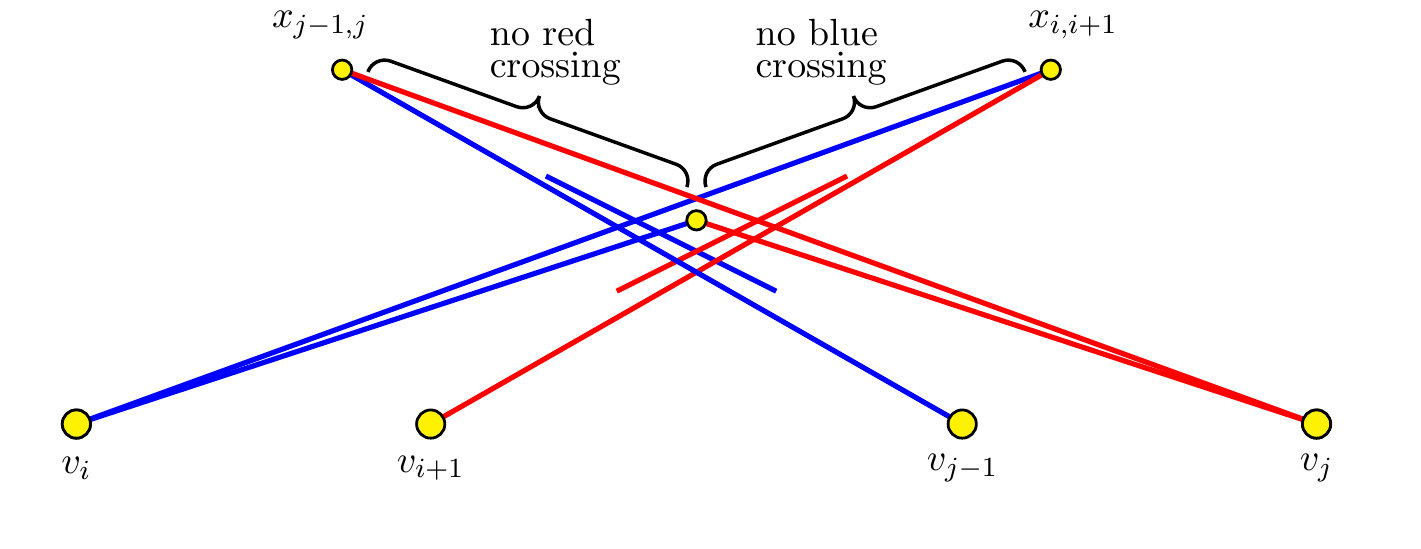}\vspace*{-5ex}}{Placing $x_{i,j}$ where $i\in[1,n-2]$ and $j\in[i+2,n]$.}

By invariant (1), no blue edge crosses $v_ix_{i,i+1}$ after the red edge $v_jx_{j-1,j}$. It follows that the new blue edge $v_ix_{i,j}$ crosses every blue edge already drawn, and invariant (1) is maintained. 
By invariant (2), no red edge crosses $v_jx_{j-1,j}$ after the blue edge $v_ix_{i,i+1}$. It follows that the new red edge $v_jx_{i,j}$ crosses every red edge already drawn, and  invariant (2) is maintained. 
\end{proof}

\citet[Lemma 10]{DujWoo05} proved that $K_n'$ has queue-number at least $\sqrt{n/6}$. Since the queue-number of a graph is at most its convex antithickness (\lemref{QNCAT}), $K_n'$ has convex antithickness at least $\sqrt{n/6}$. This proves the claim in \secref{Introduction} that implies that convex antithickness is not bounded by geometric antithickness.

%%%%%%%%%%%%%%%%%%%%%%%%%%%%%%%%%%%%%%%%%%%%%%%%%%%%%%%%
\section{Extremal Questions}\seclabel{Extremal}
%%%%%%%%%%%%%%%%%%%%%%%%%%%%%%%%%%%%%%%%%%%%%%%%%%%%%%%%

This section studies the maximum number of edges in an $n$-vertex
graph with topological (or geometric or convex or $2$-track) thickness
(or antithickness) $k$. The results are summarised in \tabref{Extremal}.
First we describe results from the literature, followed by our original results.

For book thickness and $2$-track thickness the maximum number of edges
is known.  \citet{BK79} proved that the maximum number of edges in an
$n$-vertex graph with book thickness $k$ equals
$(k+1)n-3k$. \citet{DujWoo04} proved that the maximum number of
edges in an $n$-vertex graph with $2$-track thickness $k$ equals
$k(n-k)$.

Determining the maximum number of edges in a thrackle is a famous open problem proposed by John Conway, who conjectured that every thrackle on $n$ vertices has at most $n$ edges. Improving upon previous bounds by \citet{LPS-DCG97} and \citet{CN-DCG00}, 
\citet{FP11} proved that every thrackle has at most $\frac{167}{117}n$ edges.  Thus every graph with antithickness $k$ has at most $\frac{167}{117}kn$ edges. For $n\geq 2k+1$, it is easy to construct an $n$-vertex graph consisting of $k$ edge-disjoint copies of $C_n$. Thus this graph has antithickness $k$ and $kn$ edges.

Many authors have proved that every geometric thrackle has at most $n$ edges \citep{Woodall-Thrackles,Erdos46,HopfPann34,PA95}. 
%; see \secref{GeometricThrackles}. 
Thus every graph with geometric antithickness $k$ has at most $kn$ edges. 
For convex antithickness, \citet{FW12} improved this upper bound to $kn-\binom{k}{2}$, and 
\citet{FJVW} established a matching lower bound. This lower bound is the best known lower bound in the geometric setting. It is an open problem to determine the maximum number of edges in an $n$-vertex graph with geometric antithickness $k$.

We also mention that many authors have considered graph drawings, with at most $k$ pairwise crossing edges or at most $k$ pairwise disjoint edges (instead of thickness $k$ or antithickness $k$). These weaker assumptions allow for more edges. See \citep{CapoPach-JCTB92,KupitzPerles-DCG96,Kupitz-DM84,Cerny-DCG05,AE-DCG89,GKK-EuJC96,TV-DCG99,GKK-EuJC96,Toth-JCTA00,Felsner04}.

% \comment{I suspect that this bound is not tight for $k\geq2$. }

% \begin{open} What is the maximum chromatic number $\chi$ of a
%   thickness-$k$ (a) topological graph, (b) geometric graph, (c)
%   convex graph? \end{open}

% The following bounds are known, none of which are tight for all
% $k$:\\ (a) $\chi\leq6k$,\\ (b) $\chi\leq 6k$, and\\ (c)
% $2k\leq\chi\leq2(k+1)$. \\ The case of topological graphs and $k=2$
% is the so-called \emph{earth-moon problem}---any improvement in the
% upper bound of $12$ would be of substantial interest. Case (b) seems
% the most likely candidate for progress.

% What is the maximum chromatic number $\chi$ of a $k$-queue graph?
% \citet{DujWoo04} proved that $2k+1\leq\chi\leq 4k$, and that
% $\chi=3$ for $k=1$.}

\begin{table}
  \caption{\tablabel{Extremal}The maximum number of edges in an $n$-vertex graph with given
    parameter $k$.}
  \begin{center}
    \begin{tabu}{lccr}
      \hline
      parameter  & lower bound & upper bound & reference\\\hline\\[-1ex]
      thickness $k$& $3k(n-2)$ & $3k(n-2)$ & \thmref{ExtremalThickness}\\[0.5ex]
      geometric  thickness $1$& $3n-6$ & $3n-6$ & \\
      geometric  thickness $2$& $6n-20$ & $6n-18$ & \citep{HSV-CGTA99} \\
      geometric  thickness $k$& $k(3n-4k-3)$ & $k(3n-k-5)$ &         \thmref{MaxEdgesGeomThickness} \\[0.5ex]
      book thickness $k$& $(k+1)n-3k$ & $(k+1)n-3k$ & \citep{BK79}\\
      2-track  thickness $k$& $k(n-k)$ & $k(n-k)$ & \citep{DujWoo04}\\[0.5ex]
      \hline\\[-1ex]
      antithickness $1$& $n$ & $\frac{167}{117}\,n$ & \citep{FP11}\\
      antithickness $k$& $kn$ & $\frac{167}{117}\,kn$\\[0.5ex]
      geometric antithickness $1$& $n$ & $n$ &       \citep{Woodall-Thrackles,Erdos46,HopfPann34,PA95}\\
      geometric antithickness $k$& $kn-\binom{k}{2}$ & $kn$ & \\[0.5ex]
      convex antithickness $k$& $kn-\binom{k}{2}$ & $kn-\binom{k}{2}$ & \citep{FW12,FJVW}\\
      2-track antithickness $k$& $k(n-k)$ & $k(n-k)$ & \citep{DujWoo04}\\[0.5ex]
      \hline
    \end{tabu}
  \end{center}
\end{table}

\subsection{Thickness} Since every planar graph with $n\geq3$ vertices has at most $3(n-2)$ edges, every graph with $n\geq3$ vertices and thickness $k$ has at most $3k(n-2)$ edges.  We now prove a lower bound.

% \citet{??} proved that the thickness of $K_q$ is
% $\ceil{\frac{q+2}{6}}$ for $q\neq 9,10$. So if we take $\frac{n}{q}$
% copies of $K_q$ for $q\equiv4\pmod{6}$ and $q\neq10$, we obtain a
% graph with thickness $k=\frac{q+2}{6}$ and $3n(k-\half)$ edges.

\begin{thm}
  \thmlabel{ExtremalThickness} For all $k$ and infinitely many $n$
  there is a graph with $n$ vertices, thickness $k$, and exactly
  $3k(n-2)$ edges.
\end{thm}

Let $G$ be a graph. Let $f$ be a bijection of $V(G)$. Let $G^f$ be the
graph with vertex set $V(G^f)=V(G)$ and edge set
$E(G^f)=\{f(vw):vw\in E(G)\}$, where $f(vw)$ is an abbreviation for the edge $f(v)f(w)$. 
Bijections $f_1$ and $f_2$ of $V(G)$ are \emph{compatible} if $G^{f_1}$ and $G^{f_2}$ are
edge-disjoint. By taking a union, the next lemma implies   \thmref{ExtremalThickness}. 

\begin{lem}
  \lemlabel{CompatibleBijections} For each integer $k\geq 1$ there are 
  infinitely many edge-maximal planar graphs that admits $k$ pairwise compatible  bijections.
\end{lem}

\newcommand{\uu}[1]{u\langle{#1}\rangle}
\newcommand{\vv}[1]{v\langle{#1}\rangle}
\newcommand{\ww}[1]{w\langle{#1}\rangle}

\begin{proof}
  Let $n$ be a prime number greater than $3k^2$. Let $G$ be the graph
  with vertex set $$V(G)=\{\uu{i},\vv{i},\ww{i}:i\in[0,n-1]\}$$
  and edge set $E(G)=A\cup B\cup C$, where
  \begin{align*}
    A\;=\;&\{\uu{i}\uu{i+1},\vv{i}\vv{i+1},\ww{i}\ww{i+1}: i\in[0,n-2]\}, \\
    B\;=\;  &\{\uu{i}\vv{i},\vv{i}\ww{i},\ww{i}\uu{i}: i\in[0,n-1]\}, \\
  C\;=\;    &\{\uu{i}\vv{i+1},\vv{i}\ww{i+1},\ww{i}\uu{i+1}: i \in[0,n-2]\}\enspace.
  \end{align*}
  Then $G$ is edge-maximal planar, as illustrated in
  \figref{NestedTriangles}.
  % Edges $\uu{i}\uu{i+1},\vv{i}\vv{i+1},\ww{i}\ww{i+1}$ are type-A,
  % edges $\uu{i}\vv{i},\vv{i}\ww{i},\ww{i}\uu{i}$ are type-B, edges
  % $\uu{i}\vv{i+1},\vv{i}\ww{i+1},\ww{i}\uu{i+1}$ are type-C.

  \Figure{NestedTriangles}{\includegraphics{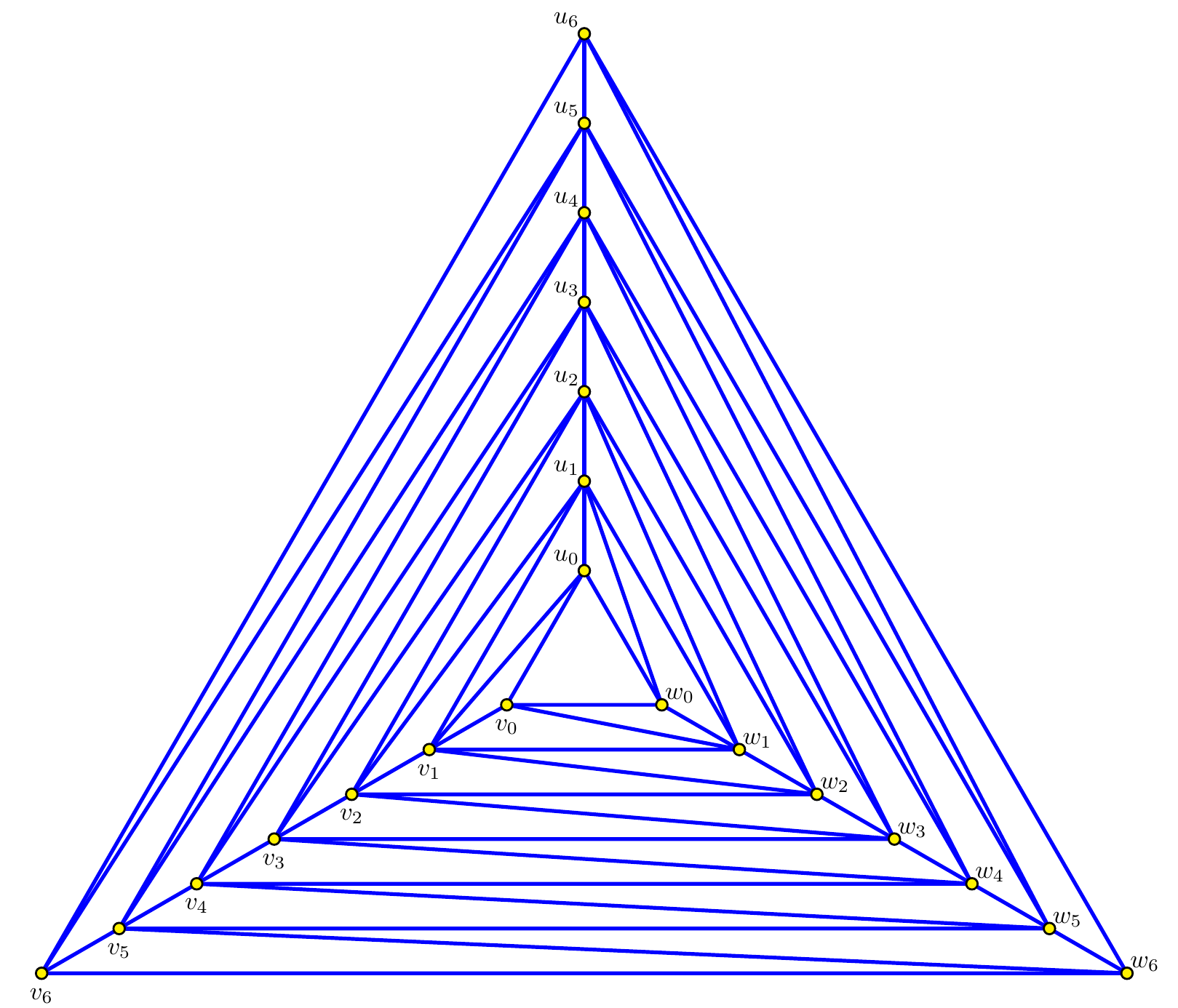}}{The
    nested triangles graph.}

%\comment{REF: In the case analysis in the proof of Lemma 19, functions $f_p$ and $f_q$
%must be defined on vertices and not on edges.  VIDA: it actually refers to the cases and not to the definition of $f_p$ and $f_q$. namely in the cases we write things like: 
%$f_p(u<i>u<i+1>)$ and $u<i>u<i+1>$ is an edge so it appears like we are applying the function $f_p$ to an edge. If we were to be precise, and in fact follow our own notation (see the part after Th 18), we should write $f_p(u<i>)f_p(u<i+1>)$, 
%or we can keep the current notation and say before the cases that with some abuse of notation, f(ab) for adjacent vertices a and b, means f(a)f(b).
%}

  For each $p\in[1,k]$, let $f_p:V(G)\rightarrow V(G)$ be the function
  defined by
  \begin{align*}
    f_p(\uu{i}) &:=\uu{pi}\\
    f_p(\vv{i}) &:=\vv{pi+p(k+1)}\\
    f_p(\ww{i})&:=\ww{pi+2p(k+1)}\enspace,
  \end{align*}
  where vertex indices are always in the cyclic group
  $\mathbb{Z}_n$. Thus $f_p$ is a bijection. 

  Suppose to the contrary that $G^{f_p}$ and $G^{f_q}$ have an edge $e$ in
  common, for some distinct $p,q\in[1,k]$.  %Without loss of generality, $p>q$.   
  Since $\{\uu{0},\dots,\uu{n-1}\}$ is mapped to $\{\uu{0},\dots,\uu{n-1}\}$
  by both $f_p$ and $f_q$, and similarly for the $\vv{i}$ and $\ww{i}$, the
  following cases suffice.   All congruences are modulo $n$.

%\comment{REF: pg 17 case distinction: I advise to give the reader a better idea
%about the structure of your cases. In the definition of E(G) you name
%3 set of edges. Say T1, T2, T3. Then Case 1 is T1---T1, case 2: T2
%---T2, case 3: T1 --- T3 and case 4 T3 ---T3.\\
%VIDA; The believe here the ref is suggesting how to break down the cases so it clear that the cases are not missed and so it is clear what 1, 2, 3, 4 stand for in the case numbering. So here is the suggestion that the ref makes (and it is not a bad one).\\
%she suggests that in the definition of E(G) at the beginning of the proof of Lemma 19, we call the edges in the first line (i<->i+1 kind) as type T1 edges, edges in the second line as type T2 edges and edges in the third line as type T3 edges. Then the bad things that can happen, but are proved impossible by these cases are:\\
%For different p and q:\\
%Case 1. two type T1 edges are the same\\
%Case 2. two type T2 edges are the same\\
%Case 3. a type T2 and a type T3 edge are the same\\
%Case 4. two type T3 edges are the same\\
%and these 4 cases match exactly the four cases that are already there.
%}

  Case 1. $e$ is from $A$ in both $G^{f_p}$ and $G^{f_q}$:

  Case 1a. $e=f_p(\uu{i}\uu{i+1})=f_q(\uu{j}\uu{j+1})$ for some $i,j$:
  Thus $\uu{pi}\uu{p(i+1)}=\uu{qj}\uu{q(j+1)}$. Then $pi\equiv qj$ and
  $pi+p\equiv qj+q$ (implying $p\equiv q$), or $pi\equiv qj+q$ and
  $pi+p\equiv qj$ (implying $pi-qj\equiv q\equiv -p$), which is a contradiction since $n>2k\geq p+q$. 

  Case 1b. $e=f_p(\vv{i}\vv{i+1})=f_q(\vv{j}\vv{j+1})$ for some $i,j$:
  Thus $\vv{pi+p(k+1)}\vv{p(i+1)+p(k+1)}=\vv{qj+q(k+1)}\vv{q(j+1)+q(k+1)}$.  
If $pi+p(k+1) \equiv qj+q(k+1)$ and $p(i+1)+p(k+1) \equiv q(j+1)+q(k+1)$, then 
$p \equiv q$, which is a contradiction  since $n>k\geq p,q$. 
Otherwise $pi+p(k+1) \equiv q(j+1)+q(k+1)$ and $p(i+1)+p(k+1) \equiv qj+q(k+1)$, implying 
$p+q \equiv 0$, which is a contradiction since $n>2k\geq p+q$.

  Case 1c. $e=f_p(\ww{i}\ww{i+1})=f_q(\ww{j}\ww{j+1})$ for some $i,j$:
  Thus $\ww{pi+2p(k+1)}\ww{p(i+1)+2p(k+1)}=\ww{qj+2q(k+1)}\ww{q(j+1)+2q(k+1)}$.  
If $pi+2p(k+1) \equiv qj+2q(k+1)$ and $p(i+1)+2p(k+1)\equiv q(j+1)+2q(k+1)$, then 
 $p \equiv q$, which is a contradiction since $n> k \geq p,q$. 
Otherwise $pi+2p(k+1) \equiv q(j+1)+2q(k+1)$ and $p(i+1)+2p(k+1) \equiv qj+2q(k+1)$, implying
$p+ q  \equiv 0$, which is a contradiction since $n>2k \geq p+q$. 

  Case 2. $e$ is from $B$ in both $G^{f_p}$ and $G^{f_q}$:

  Case 2a. $e=f_p(\uu{i}\vv{i})=f_q(\uu{j}\vv{j})$ for some $i,j$.
  Thus $\uu{pi}\vv{pi+p(k+1)}=\uu{qj}\vv{qj+q(k+1)}$.  Then $pi\equiv
  qj$ and $pi+p(k+1)\equiv qj+q(k+1)$. Hence $p(k+1)\equiv q(k+1)$ and
  $p\equiv q$ since $n$ is prime, which is a contradiction since $n> k\geq p,q$. 
 
  Case 2b. $e=f_p(\vv{i}\ww{i})=f_q(\vv{j}\ww{j})$ for some $i,j$.
  Thus $\vv{pi+p(k+1)}\ww{pi+2p(k+1)}=\vv{qj+q(k+1)}\ww{qj+2q(k+1)}$.
  Then $pi+p(k+1)\equiv qj+q(k+1)$ and $pi+2p(k+1)\equiv
  qj+2q(k+1)$. Hence $p(k+1)\equiv   q(k+1)$, implying $p\equiv q$ since $n$ is prime, which is
  a contradiction since $n> k\geq p,q$. 

  Case 2c. $e=f_p(\ww{i}\uu{i})=f_q(\ww{j}\uu{j})$ for some $i,j$.
  Thus $\ww{pi+2p(k+1)}\uu{pi}=\ww{qj+2q(k+1)}\uu{qj}$.  Then
  $pi+2p(k+1)\equiv qj+2q(k+1)$ and $pi\equiv qj$. Hence
  $2p(k+1)\equiv 2q(k+1)$, implying $p\equiv q$ since $n$ is prime,
  which is a contradiction   since $n> k\geq p,q$. 

%%%%%%
  Case 3. $e$ is from $C$ in both $G^{f_p}$ and $G^{f_q}$:

  Case 3a. $e=f_p(\uu{i}\vv{i+1})=f_q(\uu{j}\vv{j+1})$ for some $i,j$.
  Thus $\uu{pi}\vv{p(i+1)+p(k+1)}=\uu{qj}\vv{q(j+1)+q(k+1)}$.  Thus
  $pi\equiv qj$ and $p(i+1)+p(k+1)\equiv q(j+1)+q(k+1)$.  Hence
  $p(k+2)\equiv q(k+2)$, implying $p\equiv q$ since $n$ is prime, which is a contradiction  since $n> k\geq p,q$. 

  Case 3b. $e=f_p(\vv{i}\ww{i+1})=f_q(\vv{j}\ww{j+1})$ for some $i,j$.
  Thus 
  $\vv{pi+p(k+1)}\ww{p(i+1)+2p(k+1)}=\vv{qj+q(k+1)}\ww{q(j+1)+2q(k+1)}$.
  Thus $pi+p(k+1)\equiv qj+q(k+1)$ and $p(i+1)+2p(k+1)\equiv
  q(j+1)+2q(k+1)$.  Hence $p(k+2)\equiv q(k+2)$, implying $p\equiv q$ since $n$ is prime, which is a contradiction  since $n> k\geq p,q$. 

  Case 3c. $e=f_p(\ww{i}\uu{i+1})=f_q(\ww{j}\uu{j+1})$ for some $i,j$.
  Thus $\ww{pi+2p(k+1)}\uu{p(i+1)}=\ww{qj+2q(k+1)}\uu{q(j+1)}$.  Thus
  $pi+2p(k+1)\equiv qj+2q(k+1)$ and $p(i+1)\equiv q(j+1)$.  Hence
  $p(2k+1)\equiv q(2k+1)$, implying $p\equiv q$ since $n$ is prime, which is a contradiction  since $n> k\geq p,q$. 

%%%%%%%
  Case 4. $e$ is from $B$ in $G^{f_p}$ and from $C$ in $G^{f_q}$:

  Case 4a. $e=f_p(\uu{i}\vv{i})=f_q(\uu{j}\vv{j+1})$ for some $i,j$.
  Thus $\uu{pi}\vv{pi+p(k+1)}=\uu{qj}\vv{q(j+1)+q(k+1)}$.  Then
  $pi\equiv qj$ and $pi+p(k+1)\equiv q(j+1)+q(k+1)$, implying
  $(p-q)(k+1)\equiv q$. 
  If $p> q$ then $(p-q)(k+1) \in [k+1,k^2-1]$ is not congruent to $q\in[1,k]$ since $n>k^2$. 
Otherwise $p<q$, implying $(p-q)(k+1) \in [-(k+1),-(k^2-1)]$ is not congruent to $q\in[1,k]$ since $n>2k^2$. 
  
  Case 4b. $e=f_p(\vv{i}\ww{i})=f_q(\vv{j}\ww{j+1})$ for some $i,j$.
  Thus
  $\vv{pi+p(k+1)}\ww{pi+2p(k+1)}=\vv{qj+q(k+1)}\ww{q(j+1)+2q(k+1)}$.
  Then $pi+p(k+1)\equiv qj+q(k+1)$ and $pi+2p(k+1)\equiv
  q(j+1)+2q(k+1)$.  Hence $(p-q)(k+1)\equiv q$. As in Case 3a, this is
  a contradiction.

  Case 4c. $e=f_p(\ww{i}\uu{i})=f_q(\ww{j}\uu{j+1})$ for some $i,j$.
  Thus $\ww{pi+2p(k+1)}\uu{pi}=\ww{qj+2q(k+1)}\uu{q(j+1)}$.  Then
  $pi+2p(k+1)\equiv qj+2q(k+1)$ and $pi\equiv q(j+1)$.  
  Hence $pi-qj \equiv 2(q-p)(k+1)\equiv q$. 
  If $q>p$ then $2(q-p)(k+1)\in [2k+2,2k^2-2]$ is not congruent to $q\in[1,k]$ since $n>3k^2$. 
Otherwise $q<p$, implying $2(q-p)(k+1)\in [-2k^2+2,-2k-2]$ is not congruent to $q\in[1,k]$ since $n>3k^2$. 

  Therefore $f_1,\dots,f_k$ are pairwise compatible bijections of $G$.
\end{proof}

\subsection{Geometric Thickness} 

Every graph with $n\geq 3$ vertices and geometric thickness $k$ has at most $3k(n-2)$
edges. Of course, this bound is tight for $k=1$. But for $k=2$,
\citet{HSV-CGTA99} improved this upper bound to $6n-18$, and
constructed a graph with geometric thickness $2$ and $6n-20$ edges. We
have the following lower and upper bounds for general $k$. The proof
is inspired by the proofs of lower and upper bounds on the geometric
thickness of complete graphs due to \citet{DEH-JGAA00}.

\begin{thm} 
  \thmlabel{MaxEdgesGeomThickness} For $k\geq 1$ and $n\geq\max\{2k,3\}$, every graph with
  $n$ vertices and geometric thickness $k$ has at most $k(3n-k-5)$
  edges. Conversely, for all such $n\equiv0\pmod{2k}$, there is an
  $n$-vertex graph with geometric thickness $k$ and exactly $k(3n-4k-3)$
  edges.
\end{thm}

\begin{proof}[Proof of Upper Bound]
  Let $T_1,\dots,T_k$ be triangulations of $V(G)$ such that $E(G)$ is
  contained in $T_1 \cup\dots\cup T_k$.  Assume that no two vertices
  have the same x-coordinate.  Let $A$ be the set of the $k$ leftmost
  vertices. Let $B$ be the set of the $k$ rightmost vertices.  Since
  $n\geq 2k$, we have $A\cap B=\emptyset$.

  For distinct vertices $v, w \in A$, the line segment $vw$ crosses a
  number of triangular faces in $T_i$. The left sides of
  these faces form a $vw$-path in $T_i[A]$. Thus $T_i[A]$ is
  connected. Similarly $T_i[B]$ is connected.

  Thus $T_i[A]$ and $T_i[B]$ both have at least $k-1$ edges.  Hence
  $T_i$ contains at most $3n-6 - 2(k-1)=3n-2k-4$ edges with at most one
  end-vertex in $A$ and at most one end-vertex in $B$. Thus $|E(G)| \leq
  |E( T_1 \cup\dots\cup T_k ) | \leq 2 \binom{k}{2} + k (3n-2k-4) =
  k(3n - k -5)$.
\end{proof}

\begin{proof}[Proof of Lower Bound]
Fix a positive integer $s$.  We construct a geometric graph $G$ with $n=2sk$ vertices and
  geometric thickness $k$. The vertices are partitioned into levels
  $V_1,\dots,V_s$ each with $2k$ vertices, where
  $V_a:=\{(a,i):i\in[1,2k]\}$ for $a\in[1,s]$. The vertices in each level $V_a$ are
  evenly spaced on a circle $C_a$ of radius $r_a$ centred at the
  origin, where $1=r_1<\dots<r_s$ are specified below. The vertices in
  $V_a$ are ordered $(a,1),\dots,(a,2k)$ clockwise around $C_a$. Thus
  $(a,j)$ is opposite $(a,k+j)$, where the second coordinate is always modulo $2k$. All congruences below are modulo $2k$. 

  The first level $V_1$ induces a complete graph. For distinct 
  $i,j\in[1,2k]$, the edge $(1,i)(1,j)$ is coloured by the
  $\ell\in[1,k]$ such that $i+j \equiv 2\ell$ or $i+j\equiv 2\ell-1$. The edges coloured $\ell$
  form a non-crossing path with end-vertices $(1,\ell)$ and $(1,k+\ell)$,
  as illustrated in \figref{BookEmbedding}. Note that
  $(1,\ell+\floor{\frac{k}{2}})(1,\ell+\floor{\frac{3k}{2}})$ is the
  `long' edge in this path.  This is a well-known construction of a
  $k$-page book embedding of $K_{2k}$; see \citep{BHRW06} for
  example. This contributes $\binom{2k}{2}$ edges to $G$.

  \Figure{BookEmbedding}{\includegraphics{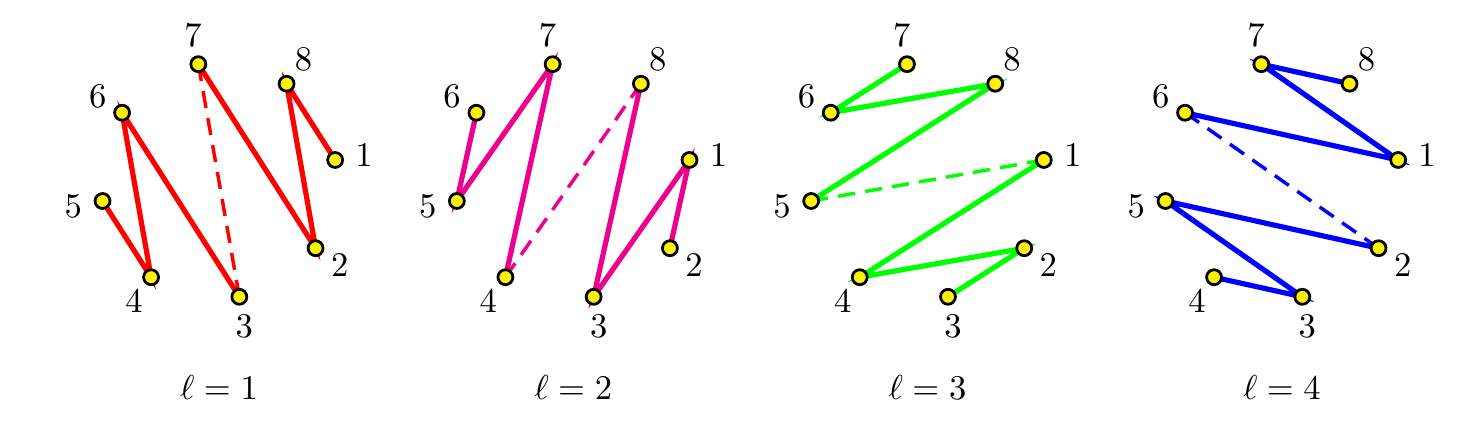}}{The edges
    between vertices in $V_a$ with $k=4$. The dashed edges are
    included only for $a=1$.}

  Every other level $V_a$ (where $a\in[2,s]$) induces a complete graph
  minus a perfect matching. We use a partition into non-crossing
  paths, analogous to that used in the $a=1$ case, except that the
  `long' edge in each path is not included. More precisely, for distinct 
  $i,j\in[1,2k]$ with $i\not\equiv k+j$, the edge $(a,i)(a,j)$ is
  coloured by the $\ell\in[1,k]$ such that $i+j \equiv 2\ell$ or
  $i+j \equiv 2\ell-1$. The edges coloured $\ell$ form two non-crossing
  paths, one with end-vertices $(a,\ell)$ and
  $(a,\ell+\floor{\frac{3k}{2}})$, the other with end-vertices
  $(a,\ell+\floor{\frac{k}{2}})$ and $(a,k+\ell)$, as illustrated in
  \figref{BookEmbedding}.  This contributes $(s-1)(\binom{2k}{2}-k)$
  edges to $G$.

  We now define the edges between the layers, as illustrated in
  \figref{ThreeLayerTwoLevel}. For $a\in[2,s]$ and $i\in[1,k]$ and
  $j\in[1,2k]$, the edges $(a,i+\floor{\frac{k}{2}})(a-1,j)$ and
  $(a,i+\floor{\frac{3k}{2}})(a-1,j)$ are present and are coloured
  $i$. This contributes $2(s-1)2k^2$ edges to $G$.  Finally, for
  $a\in[2,s-1]$ and $i\in[1,k]$, the edges
  $(a+1,i+\floor{\frac{k}{2}})(a-1,k+i)$ and
  $(a+1,i+\floor{\frac{3k}{2}})(a-1,i)$ are present and are coloured
  $i$.  This contributes $2(s-2)k$ edges to $G$.

  As illustrated in \figref{ThreeLayerTwoLevel}, given a drawing of
  the first $a$ layers (which are defined by $r_1,\dots,r_a$) there is
  a sufficiently large value of $r_{a+1}$ such that the addition of
  the $(a+1)$-th layer does not create any crossings between edges
  with the same colour.

  In total, $G$ contains
  $\binom{2k}{2}+(s-1)(\binom{2k}{2}-k)+  (s-1)(\binom{2k}{2}-k) + 2(s-2)k$ edges, which equals $6ks-4k^2-3k=k(3n-4k-3)$.  
%\comment{REF:  I am puzzled how you come to this formula. I, for one,
%would count the edges in one color class: then we have $2k-1 +
%(s-1)(2k-2) + (s-1)4k + 2(s-2)$ edges with gives the desired
%bound. }
%
\end{proof}

\Figure[!t]{ThreeLayerTwoLevel}{\includegraphics{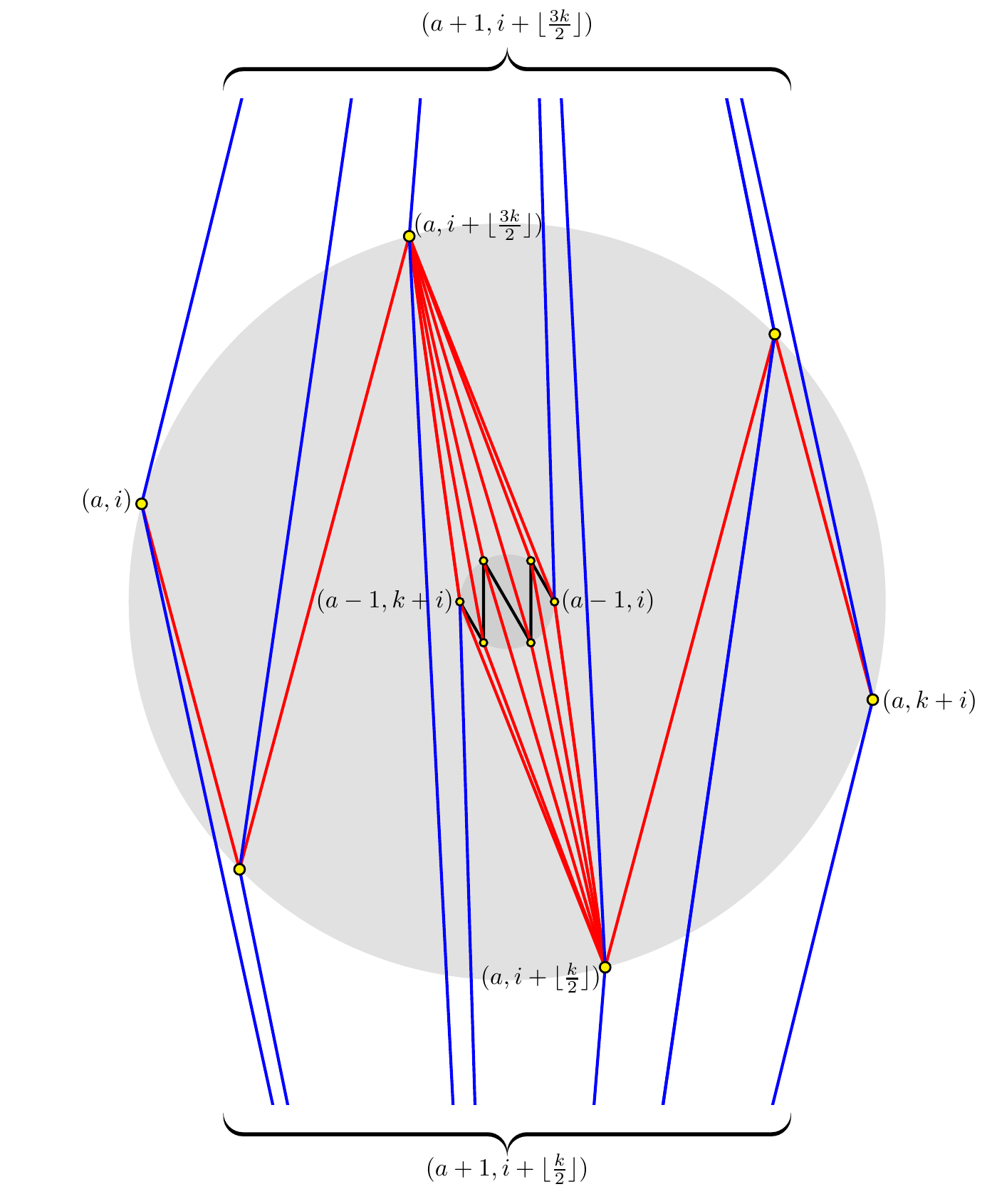}}{Edges coloured $i$ between layers.}

%\Figure[!t]{Construction}{\includegraphics{Construction}}{Edges  coloured $i$ between layers.}

% The construction is defined recursively. For the base case $n=2k$,
% we use a well known construction of a $k$-page book embedding of
% $K_n$, in which each page contains a spanning path. For
% $\ell\in[1,k]$, let $E_{\ell}$ be the set of edges $v_iv_j$ such
% that $i+j\equiv2\ell\pmod{n}$ or $i+j\equiv2\ell-1\pmod{n}$. No two
% edges in $E_{\ell}$ cross, and every edge of $K_n$ is in exactly one
% $E_{\ell}$. The number of edges is $\binom{n}{2}=3kn-????k^2-???k$.

% The edges of $G$ come in two flavours: \emph{inter-level} edges and
% \emph{intra-level} edges.

% The edges of $G_s$ are partitioned into plane subgraphs
% $G_{s,1},\dots,G_{s,k}$, such that for each $i\in[k]$, there are
% antipodal points $x_i$ and $y_i$ such that in $G_{s,i}$, both $x_i$
% and $y_i$ can see vertex $v_{s,b}$ for all $b\in[2k]$.

% We proceed by induction on $s\geq1$ with the following hypothesis:
% There is an geometric graph $G_s$ with $2sk$ vertices denoted by
% $$V(G_s):=\{v_{a,b}:1\leq a\leq s,0\leq b\leq 2k-1\},$$ 
% such that $G_s$ has an edge-partition into plane subgraphs
% $G_{s,1},\dots,G_{s,k}$, and for each $i\in[k]$, there are antipodal
% points $x_i$ and $y_i$ such that in $G_{s,i}$, both $x_i$ and $y_i$
% can see vertex $v_{s,b}$ for all $b\in[2k]$.

% As illustrated in \twofigref{TwoLayer}{ThreeLayer} each colour class
% has $(3n-6)-(2k-3)-(2k-2)-4=3n-4k-5$ edges.

Examples of the construction in \thmref{MaxEdgesGeomThickness} are
given in \twofigref{TwoLayer}{ThreeLayer}.

\Figure[!h]{TwoLayer}{\includegraphics[width=\textwidth]{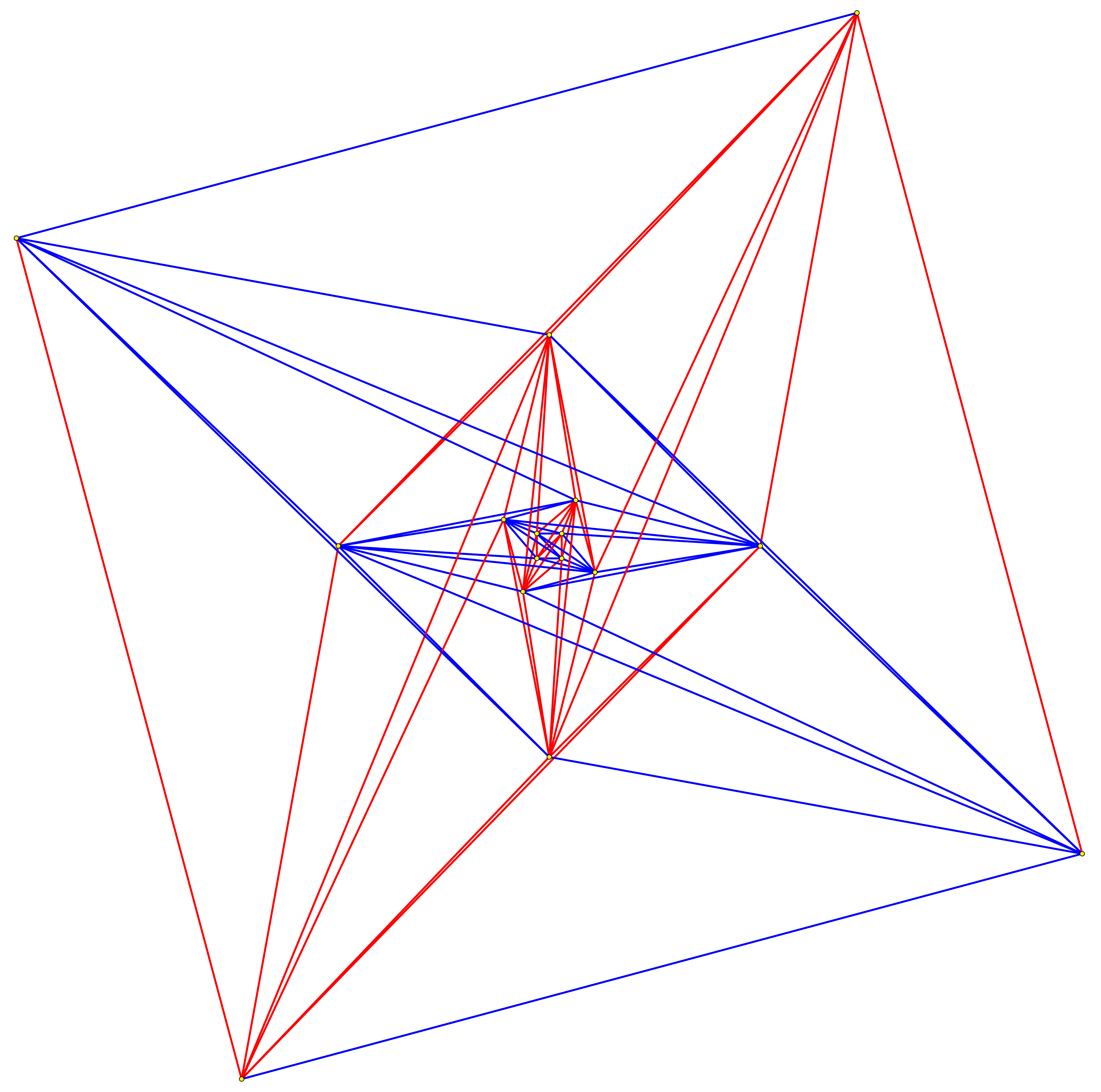}}{The
  construction in \thmref{MaxEdgesGeomThickness} with $k=2$ and
  $s=4$.}

\Figure[!h]{ThreeLayer}{\includegraphics[width=\textwidth]{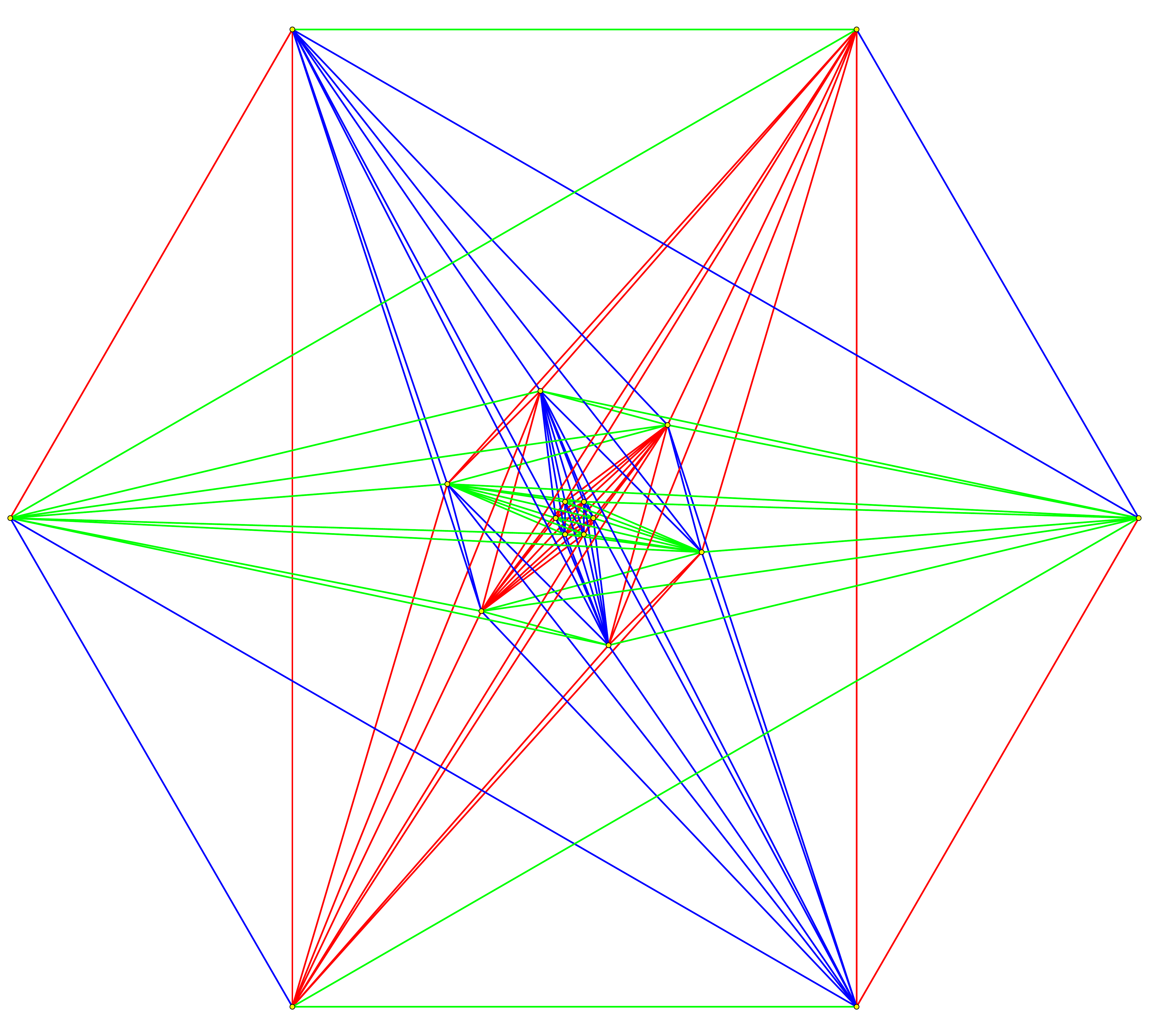}}{The
  construction in \thmref{MaxEdgesGeomThickness} with $k=3$ and
  $s=3$.}

\section{Antithickness of Complete Graphs}\seclabel{Complete}
%%%%%%%%%%%%%%%%%%%%%%%%%%%%%%%%%%

Let $\ctn{G}$ be the convex antithickness of a graph $G$. We now consider $\ctn{K_n}$. \citet{ADHNU05} proved\footnote{\citet{ADHNU05} did not use the terminology of `antithickness', but it is easily seen that their definition of $d_c(n)=\ctn{K_n}$.} that 
\begin{equation}
  \eqnlabel{ADHNU}
  2\FLOOR{\frac{n+1}{3}}-1\leq\ctn{K_n}<n-\HALF\FLOOR{\log n}.
\end{equation}
In the original version of this paper (cited in \citep{FJVW,FW12}), we improved both the lower and upper bound as follows.
\begin{thm}
  \thmlabel{ConvexCompleteAntithickness} 
The convex antithickness of the complete graph $K_n$ satisfies
$$\frac{3n-6}{4}\leq\ctn{K_n}<n-\sqrt{\frac{n}{2}}-\frac{\ln n}{2}+4.$$
\end{thm}
In 2007, we conjectured that $\ctn{K_n}=n-o(n)$. This conjecture was subsequently verified by \citet{FW12} who proved that every $n$-vertex graph with convex antithickness $k$ has at most $kn-\binom{k}{2}$ edges, which implies that 
\begin{equation}
\eqnlabel{FabWoo}
\ctn{K_n}\geq n-\sqrt{2n+\quarter}+\half.
\end{equation}
This is a significant improvement over the lower bound in \thmref{ConvexCompleteAntithickness}. The upper bound in \thmref{ConvexCompleteAntithickness} has since been improved by \citet{FJVW} to match the lower bound in \eqnref{FabWoo}. Thus
\begin{equation*}
\ctn{K_n}= n-\FLOOR{\sqrt{2n+\quarter}-\half}.
\end{equation*}
For the historical record we include our proof of \thmref{ConvexCompleteAntithickness} in an appendix.
% / in the arXiv version of this paper. 

%\comment{DW: Put reference \cite{FJVW}  on the arXiv before resubsmission of the present paper. This paper is now finished, just waiting on the other authors to read it. }
  
% SURVEY THE THICKNESS OF COMPLETE GRAPHS.  In this section we
% consider the antithickness of complete graphs.

%\medskip \citet{ADHNU05} proved that there is a drawing of $K_n$
%with antithickness at least $5\floor{\frac{n}{7}}$, and every drawing
%of $K_n$ has antithickness at most $n+\half-\half\floor{\log\log
%  n}$. The proof of the upper bound uses the Erd\H{o}s-Szekeres
%theorem in conjunction with \eqnref{ADHNU}. Using the same trick in
%conjunction with \thmref{ConvexCompleteAntithickness} we obtain the
%following improved upper bound.
%
%\begin{proposition}
%  Every geometric drawing of $K_n$ has antithickness at most
%$$n-\half\sqrt{\log n}-\half\ln(\half\log n)+4.$$
%\end{proposition}
%
%\begin{proof}
%  \citet{ES35} proved that every set of $n$ points in general position
%  contains a subset $S$ of $m\geq\half\log n$ points in convex
%  position. By \thmref{ConvexCompleteAntithickness} the complete
%  convex graph on $S$ has antithickness less than
%  $m-\sqrt{\frac{m}{2}}-\frac{\ln m}{2}+4$. For each vertex $v$ not in
%  $S$ introduce a new colour $c_v$ and colour every edge incident to
%  $v$ by $c_v$. The total number of colours is less than
%  $(n-m)+(m-\sqrt{\frac{m}{2}}-\half\ln m+4)
%  =n-\sqrt{\frac{m}{2}}-\half\ln m+4 <n-\sqrt{\quarter\log
%    n}-\half\ln(\half\log n)+4 =n-\half\sqrt{\log
%    n}-\half\ln(\half\log n)+4$.
%\end{proof}
%

Now consider the antithickness of $K_n$.

\begin{prop}
  The antithickness of $K_n$ is at least $\frac{n}{3}$ and at most
  $\ceil{\frac{n-1}{2}}$.
\end{prop}

\begin{proof}
The lower bound follows from the fact that every graph with   antithickness at most $k$ has at most $\frac{3}{2}k(n-1)$ edges; see   \secref{Extremal}. For the upper bound, first consider the case of odd $n$. 
Walecki proved $K_n$ has a edge-partition into $\frac{n-1}{2}$ Hamiltonian cycles~\citep{Lucas1892}. Each such cycle is a thrackle. By \corref{AntithicknessCharacterisation},  the antithickness of $K_n$ is at most $\frac{n-1}{2}$. For even $n$, (applying the odd case) there is an edge-partition into $\frac{n-2}{2}$ odd cycles of length $n-1$, plus one $(n-1)$-edge star. Each such cycle and the star is a thrackle. By \corref{AntithicknessCharacterisation},  the antithickness of $K_n$ is at most $\frac{n-2}{2}+1=\ceil{\frac{n-1}{2}}$. 
\end{proof}

%It is well known that $K_n$ can be partitioned   into $\ceil{\frac{n}{2}}$ paths. \citet{Woodall-Thrackles} proved   that every tree is a thrackle. By \corref{AntithicknessCharacterisation},  the antithickness of $K_n$ is at
%  most $\ceil{\frac{n}{2}}$.

We conjecture that the antithickness of $K_n$ equals $\ceil{\frac{n-1}{2}}$ (which is implied by Conway's thrackle conjecture). Determining the geometric antithickness of $K_n$ is an open problem. The best known upper  bound is $n-\floor{\sqrt{2n+\quarter}-\half}$, which follows  from the convex case. The
  best lower bound is only $\frac{n-1}{2}$, which follows from the
  fact that every $n$-vertex graph with geometric antithickness $k$
  has at most $kn$ edges.

\subsection*{Acknowledgement}

This research was  initiated at the 2006 Bellairs Workshop on Computational Geometry organised by Godfried Toussaint. Thanks to the other workshop participants for creating a productive working environment. 

%%%%%%%%%%%%%%%%%%%%%%%%%%%%%%%%%%%%%%%%%%%%%%%%%%%%%%%%
  \let\oldthebibliography=\thebibliography
  \let\endoldthebibliography=\endthebibliography
  \renewenvironment{thebibliography}[1]{%
    \begin{oldthebibliography}{#1}%
      \setlength{\parskip}{0.2ex}%
      \setlength{\itemsep}{0.2ex}%
  }{\end{oldthebibliography}}
%\bibliographystyle{myNatbibStyle}
%\bibliography{myBibliography}

\def\soft#1{\leavevmode\setbox0=\hbox{h}\dimen7=\ht0\advance \dimen7
  by-1ex\relax\if t#1\relax\rlap{\raise.6\dimen7
  \hbox{\kern.3ex\char'47}}#1\relax\else\if T#1\relax
  \rlap{\raise.5\dimen7\hbox{\kern1.3ex\char'47}}#1\relax \else\if
  d#1\relax\rlap{\raise.5\dimen7\hbox{\kern.9ex \char'47}}#1\relax\else\if
  D#1\relax\rlap{\raise.5\dimen7 \hbox{\kern1.4ex\char'47}}#1\relax\else\if
  l#1\relax \rlap{\raise.5\dimen7\hbox{\kern.4ex\char'47}}#1\relax \else\if
  L#1\relax\rlap{\raise.5\dimen7\hbox{\kern.7ex
  \char'47}}#1\relax\else\message{accent \string\soft \space #1 not
  defined!}#1\relax\fi\fi\fi\fi\fi\fi}

%%%%%%%%%%%%%%%%%%%%%%%%%%%%%%%%%%%%%%%%%%%%%%%%%%%%%%%%
\appendix
%%%%%%%%%%%%%%%%%%%%%%%%%%%%%%%%%%%%%%%%%%%%%%%%%%%%%%%%
\section{Proof of \thmref{ConvexCompleteAntithickness}}
%%%%%%%%%%%%%%%%%%%%%%%%%%%%%%%%%%%%%%%%%%%%%%%%%%%%%%%%

\begin{proof}[Proof of \thmref{ConvexCompleteAntithickness} Lower Bound]
  Say $n$ is the minimum counterexample. That is,
  $\ctn{K_n}<\frac{3n-6}{4}$, but $\ctn{K_m}\geq\frac{3m-6}{4}$ for
  all $m<n$. It is easily seen that $n\geq5$. Consider a convex
  drawing of $K_n$ with its edges coloured with $\ctn{K_n}$ colours,
  such that monochromatic edges intersect. Consider four consecutive
  vertices $1,2,3,4$ on the convex hull. Say the edges $12$, $23$,
  $34$ are respectively coloured $a,b,c$. Since $12$ and $34$ do not
  intersect, $a\ne c$. Suppose to the contrary that $a\ne b$ and $b\ne
  c$. Then $a,b,c$ are distinct colours. Every edge coloured $a$, $b$
  or $c$ is incident to $1$, $2$, $3$ or $4$. Thus
  $K_n-\{1,2,3,4\}$ receives at most $\ctn{K_n}-3$
  colours. Hence $$\frac{3n-18}{4}=\frac{3(n-4)-6}{4}\leq\ctn{K_{n-4}}\leq\ctn{K_n}-3<\frac{3n-6}{4}-3=\frac{3n-18}{4}.$$
  This contradiction proves that $a=b$ or $b=c$. Hence, as illustrated
  in \figref{LowerBound}, the sequence of colours on the boundary
  edges\footnote{ An edge contained in the convex hull of a convex
    drawing is a \emph{boundary} edge; otherwise it is
    \emph{internal}.} of $K_n$ is $a,a,b,b,c,c,\ldots$. In particular,
  there are $\frac{n}{2}$ colours on the boundary edges (and $n$ is
  even). Let $S$ be the set of vertices for which the two incident
  boundary edges are monochromatic. Then $|S|=\frac{n}{2}$. An
  internal edge of $K_n- S$ cannot receive a colour that is
  assigned to a boundary edge of $K_n$. There are
  $\FLOOR{\half(\frac{n}{2}-2)}$ pairwise disjoint internal edges of
  $K_n- S$. Thus the number of colours is at least
  $\frac{n}{2}+\FLOOR{\half(\frac{n}{2}-2)}\geq\frac{n}{2}+\half(\frac{n}{2}-3)=\frac{3n-6}{4}$,
  as desired. \end{proof}

\Figure{LowerBound}{\includegraphics{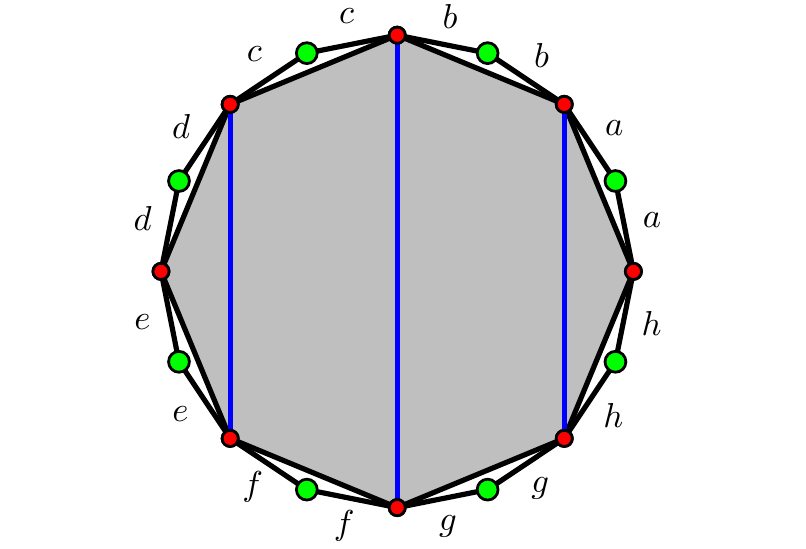}}{Illustration for the
  proof of the lower bound in \thmref{ConvexCompleteAntithickness}.}

% \begin{lem}
%   \lemlabel{OldComplete} Every complete graph $K_n$ satisfies
%   $\ctn{K_n}\leq n-2$.
% \end{lem}
%
%\begin{proof}
%  Let $(v_1,v_2,\dots,v_n)$ be the vertex ordering of $K_n$.  Let
%  $T_1$ be the set of edges $\{v_2v_n,v_1v_i:2\leq i\leq n\}$.  Then
%  $T_1$ consists of a star plus one edge that crosses every edge of
%  the %star.
%  Thus $T_1$ is a thrackle.  For each $3\leq j\leq n-1$, let $T_j$ be
%  the set of edges $\{v_jv_i:1\leq i\leq n,i\ne j\}$.  Each $T_j$ is
%  a star, and is hence a thrackle.  Clearly every edge of $K_n$ is in
%  at least one of our $n-2$ thrackles.  Thus $\ctn{K_n}\leq n-2$.
%\end{proof}

We prove the upper bound in \thmref{ConvexCompleteAntithickness} by a
series of lemmas. To facilitate an inductive argument, it will be
beneficial to prove an upper bound for a broader class of convex
graphs than just $K_n$. Let $(0,1,\dots,{n-1})$ be the vertices in
clockwise order around a convex $n$-gon. The \emph{distance} between
vertices $i$ and $j$, denoted by $\dist(i,j)$, is the length of the
shortest path between $i$ and $j$ around the boundary. That is,
$\dist(i,j)=\min\{|j-i|,n-|j-i|\}$. The \emph{distance} of an edge is
the distance between its end-vertices. Let $G(n,\ell)$ be the convex
graph on $n$ vertices, where $ij$ is an edge if and only if
$\dist(i,j)\geq\ell$. Since every distance is at most
$\floor{\frac{n}{2}}$, the only interesting case is
$\ell\in[1,\floor{\frac{n}{2}}]$. Observe that $K_n=G(n,1)$. We
proceed by upward induction on $\ell$ and downward induction on
$n$. For the base case, observe that $G(n,\floor{\tfrac{n}{2}})$ is a
thrackled perfect matching or odd cycle; thus
$\ctn{G(n,\floor{\tfrac{n}{2}})}=1$.
% , and
%$$\ctn{G(n,\floor{\tfrac{n}{2}}-1)}=
%\begin{cases}
%  2	&	\text{if $n$ is even,}\\
%  3 & \text{if $n$ is odd.}
%\end{cases}
%$$
In general, we have the following recursive construction.

\begin{lem}
  \lemlabel{ConvexConstruction}
  % For all integers $1\leq\ell\leq\floor{\frac{n}{2}}-1$, \\
  For all integers $n\geq1$ and $\ell\geq1$,
  \begin{equation*}
    \ctn{G(n,\ell)}\;\leq\;
    \CEIL{\frac{n}{\ell+1}}+\ctn{G(\FLOOR{\frac{\ell n}{\ell+1}},\ell+1)}\enspace.
  \end{equation*}
\end{lem}

\begin{proof} The lemma is vacuous for all
  $\ell\geq\floor{\frac{n}{2}}+1$. Now assume that
  $\ell\in[1,\floor{\frac{n}{2}}]$. Let $S$ be a set of $\ceil{\frac{n}{\ell+1}}$ vertices with $\ell-1$ or $\ell$ vertices
  not in $S$ between each pair of consecutive vertices in $S$. Observe
  that $|S|\geq2$. With each vertex $v\in S$, associate a distinct
  colour $c_v$. Assign the colour $c_v$ to all edges incident to $v$,
  and to all edges $xy$ of distance $\ell$ or $\ell+1$, such that a
  shortest path between $x$ and $y$ around the boundary passes through
  $v$. Observe that the edges coloured $c_v$ form a thrackle, as
  illustrated in \figref{ConvexConstruction}(a). Moreover, every edge
  of $G(n,\ell)$ that is incident to a vertex in $S$ is coloured, and
  every edge of distance at most $\ell+1$ in $G(n,\ell)$ is coloured,
  as illustrated in \figref{ConvexConstruction}(b). The number of
  vertices incident to an uncoloured edge is
  $n-|S|=n-\ceil{\frac{n}{\ell+1}}=\floor{\frac{\ell
      n}{\ell+1}}$. Every uncoloured edge has distance at least
  $\ell+2$. An uncoloured with distance $\ell+t$ spans at most $t-1$
  vertices in $S$. Thus after deleting $S$, the uncoloured edges form
  $G(\floor{\frac{\ell n}{\ell+1}},\ell+1)$, as illustrated in
  \figref{ConvexConstruction}(c).
\end{proof}

% Say $u,v,w$ are consecutive vertices in $S$.  $x$ is between $u$ and
% $v$, $y$ is between $v$ and $w$, and

\Figure{ConvexConstruction}{\includegraphics{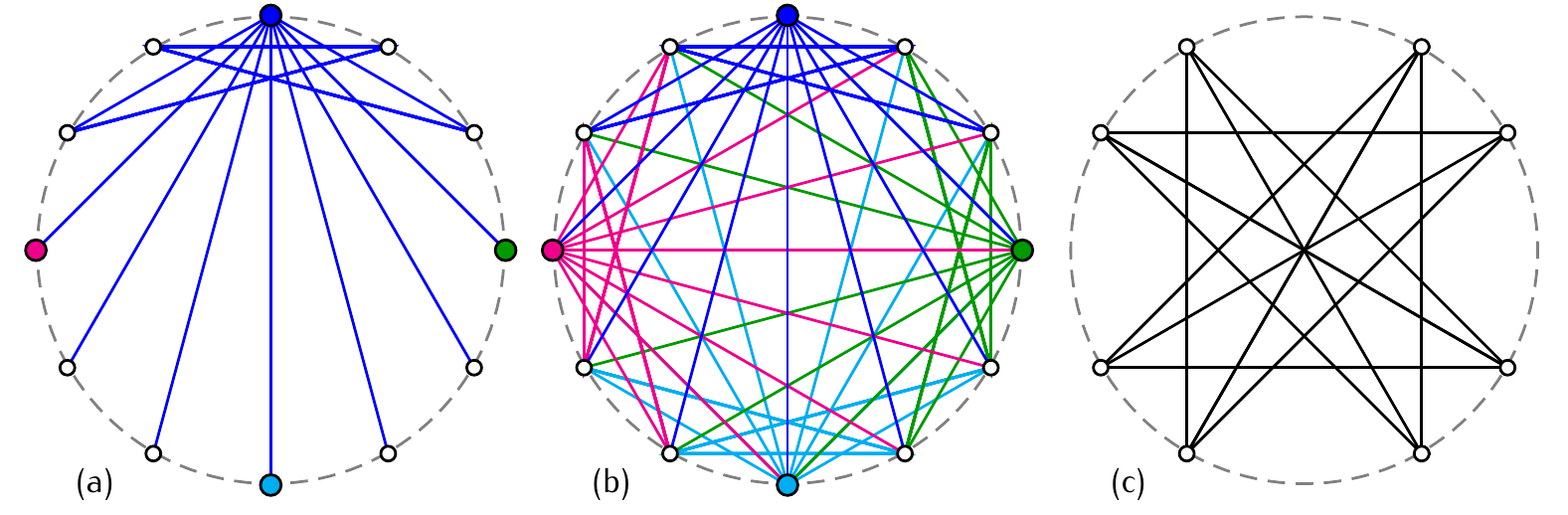}}{Colouring
  $G(12,2)$. (a) one thrackle, (b) four thrackles, (c) uncoloured
  $G(8,3)$.}

%Let $H(n):=\sum_{i=1}^n\frac{1}{i}$ denote the $n$-th harmonic
%number. It is well-known that
%\begin{equation}
%  \eqnlabel{Harmonics}
%  \ln n+\gamma
%  \leq 
%  H(n)
%  \leq
%  \ln n+\gamma+\frac{1}{12n}
%  <
%  1+\ln n,
%\end{equation}
%where $\gamma=0.577\ldots$ is Euler's constant, and $\ln n$ is the
%natural logarithm of $n$; see \citep{HarmonicNumber}.

Recall that $H(n):=\sum_{i=1}^n\frac{1}{i} \geq \ln n+\gamma$. 

\begin{lem}
  \lemlabel{ConvexInduction}
  % For all $n\geq 2$ and $1\leq\ell\leq\ceil{\sqrt{\frac{n}{2}}}$\\
  % For all integers $\ell\geq1$ and $n\geq2\ell^2$,
  For all integers $n\geq\ell\geq1$,
  \begin{equation*}
    \ctn{K_n}\;\leq\;n-\frac{n}{\ell}+(\ell-H(\ell))
    +\ctn{G(\FLOOR{\frac{n}{\ell}},\ell)}\enspace.
  \end{equation*}
\end{lem}

\begin{proof}
  We proceed by induction on $\ell$. For $\ell=1$, the lemma claims
  that $\ctn{K_n}\leq\ctn{G(n,1)}$, which holds with equality.
  % Suppose that the lemma holds for $\ell-1<\ceil{\sqrt{\frac{n}{2}}}$. \\
  % Suppose that the lemma holds for $\ell-1$ and $n\geq2(\ell-1)^2$. \\
  Suppose that $\ell\geq2$ and the lemma holds for $\ell-1$.  That is,
  \begin{equation*}
    \ctn{K_n}\;\leq\;
    n-\frac{n}{\ell-1}+(\ell-1-H(\ell-1))
    +\ctn{G(\FLOOR{\frac{n}{\ell-1}},\ell-1)}\enspace.
  \end{equation*}
  By \lemref{ConvexConstruction} applied to
  $G(\floor{\frac{n}{\ell-1}},\ell-1)$ we have
  \begin{align*}
    \ctn{K_n} &\;\leq\; n-\frac{n}{\ell-1}+(\ell-1-H(\ell-1))
    +\CEIL{\FLOOR{\frac{n}{\ell-1}}/\ell}
    +\ctn{G(\FLOOR{(\ell-1)\FLOOR{\frac{n}{\ell-1}}/\ell},\ell)}\\
    &\;\leq\; n-\frac{n}{\ell-1}+\frac{n}{\ell(\ell-1)}
    +\bracket{\ell-1-H(\ell-1)+\frac{\ell-1}{\ell}}
    +\ctn{G(\FLOOR{\frac{n}{\ell}},\ell)}\\
    &\;=\; n-\frac{n}{\ell} +\bracket{\ell-H(\ell)}
    +\ctn{G(\FLOOR{\frac{n}{\ell}},\ell)}\enspace.\qedhere
  \end{align*}
\end{proof}

\begin{proof}[Proof of \thmref{ConvexCompleteAntithickness} Upper Bound]
  Apply \lemref{ConvexInduction} with
  $\ell=\ceil{\sqrt{\frac{n}{2}}}$. Observe that
  $\ell\geq\half\floor{\frac{n}{\ell}}$. Thus
  $\ctn{G(\FLOOR{\frac{n}{\ell}},\ell)}\leq 1$, and
  \begin{align*}
    \ctn{K_n} &\;\leq\; n-\frac{n}{\ceil{\sqrt{\tfrac{n}{2}}}}
    +\CEIL{\sqrt{\tfrac{n}{2}}}- H\left(\CEIL{\sqrt{\tfrac{n}{2}}}\right) +1\\
    &\;<\;
    n- \left(\sqrt{2n}-2 \right)+ \left(\sqrt{\tfrac{n}{2}}+1\right) - H\left(\CEIL{\sqrt{\tfrac{n}{2}}}\right)+1\\
    &\;\leq\;
    n-\sqrt{\tfrac{n}{2}} - \left(\ln\CEIL{\sqrt{\tfrac{n}{2}}} + \gamma\right)+4\\
    &\;\leq\;
    n-\sqrt{\tfrac{n}{2}}- \left(\half\ln(\tfrac{n}{2}\right)+\gamma)+4\\
    &\;<\; n-\sqrt{\tfrac{n}{2}}-\half\ln n+4 \enspace.\qedhere
  \end{align*}
\end{proof}

\end{document}